\documentclass[12pt]{article}
\usepackage{amsmath,amssymb,amsthm,bbm,natbib,color,hyperref,caption,graphicx,subcaption,afterpage}
\usepackage[margin=1.2in]{geometry}
\usepackage[hang,flushmargin]{footmisc}
\usepackage[table,xcdraw]{xcolor}
\usepackage{tikz}
\usepackage{tkz-euclide}

\newtheorem{theorem}{Theorem}
\newtheorem{lemma}{Lemma}
\newcommand{\R}{\mathbb{R}}
\newcommand{\eps}{\epsilon}

\newcommand{\Ep}[2]{\mathbb{E}_{{#1}}\left[{#2}\right]}

\newcommand{\PP}[1]{\mathbb{P}\left\{{#1}\right\}}
\newcommand{\Pp}[2]{\mathbb{P}_{{#1}}\left\{{#2}\right\}}

\newcommand{\Ppst}[3]{\mathbb{P}_{{#1}}\left\{{#2}\  \middle| \ {#3}\right\}}

\newcommand{\One}[1]{{\mathbbm{1}}\left\{{#1}\right\}}
\newcommand{\iidsim}{\stackrel{\textnormal{iid}}{\sim}}

\newcommand{\Chn}{\widehat{C}_n}

\newcommand{\traindata}{\mathcal{D}_n}
\newcommand{\littleo}{\mathrm{o}}
\newcommand{\alg}{\mathcal{A}}
\newcommand{\muh}{\widehat{\mu}}
\newcommand{\Xcal}{\mathcal{X}}
\newcommand{\Ecal}{\mathcal{E}}
\renewcommand{\mod}{\textnormal{mod}}

\newcommand*\samethanks[1][\value{footnote}]{\footnotemark[#1]}

\graphicspath{{./figs/}}

\title{Training-conditional coverage for distribution-free predictive inference}
\author{Michael Bian\thanks{Department of Statistics, University of Chicago} \and
 Rina Foygel Barber\samethanks}
\begin{document}
\maketitle

\begin{abstract}
The field of distribution-free predictive inference provides tools for provably
valid prediction without any assumptions on the distribution of the data, 
which can be paired with any regression algorithm to provide accurate and reliable
predictive intervals. The guarantees provided by these methods are typically marginal,
meaning that predictive accuracy holds on average over both the training data set
and the test point that is queried. However, it may be preferable to obtain 
a stronger guarantee of training-conditional coverage, which would ensure that 
most draws of the training data set result in accurate predictive accuracy on future
test points. This property is known to hold for the split conformal prediction
method. In this work, we examine the training-conditional coverage properties of several other
 distribution-free predictive inference methods, and find
that training-conditional coverage is achieved by some methods but is impossible to guarantee
without further assumptions for others.
\end{abstract}

\section{Introduction}\label{sec:intro}
Distribution-free predictive inference provides a set of methods for constructing predictive confidence intervals with minimal assumptions about the underlying distribution. Specifically, in the case of regression, suppose we are given i.i.d. training points $(X_i, Y_i) \in \Xcal \times \R$, $i=1,...,n$, and a regression algorithm $\alg$ that maps training points to prediction rules $\muh: \Xcal\to \R$.
Given a new feature vector $X_{n+1}$, we would like to predict the unseen response $Y_{n+1}$.
 A predictive interval $\Chn$, trained on these $n$ training data points using this regression algorithm $\alg$, 
returns an interval (or more generally, a subset) $\Chn(X_{n+1})\subseteq\R$, with the goal that  $\Chn(X_{n+1})$
should contain the response value $Y_{n+1}$ for this test point.
We say that $\Chn$ is a {\em distribution-free predictive interval} if, for every distribution $P$ on $\Xcal\times\R$
it holds that
\begin{equation}\label{eqn:DF_marginal}
\Pp{P^{n+1}}{ Y_{n+1} \in \Chn(X_{n+1}) } \geq 1-\alpha.
\end{equation}
Here the notation $\Pp{P^{n+1}}{\cdot}$ 
denotes that the probability is computed with respect to $(X_1, Y_1),...,(X_n, Y_n), (X_{n+1}, Y_{n+1})\iidsim P$.

In practice, we are often interested in the coverage rate for test points once we fit a regression algorithm to a particular training set. However, the guarantee in \eqref{eqn:DF_marginal} does not directly address this. Rather, it bounds the miscoverage rate \textit{on average} over possible sets of training data and test points. As a result, if there is high variability in the coverage rate as a function of the training data, the test coverage rate may be substantially below $1-\alpha$ for a particular training set. In this case, while \eqref{eqn:DF_marginal} is satisfied on average, after fitting on the realized draw of the training set distribution the practitioner may be left with  prediction intervals which drastically undercover.

To formalize this intuition, let $\traindata =\big( (X_1, Y_1), ...,  (X_n,Y_n)\big)$ be the training data set.
 Then, define the miscoverage rate as a function of the training data:
\[\alpha_P(\traindata) = \Ppst{P}{ Y_{n+1} \notin \Chn(X_{n+1}) }{ \traindata},\]
where the probability is now only with respect to the test point $(X_{n+1}, Y_{n+1})$ drawn from $P$. Then, the guarantee in \eqref{eqn:DF_marginal} can be re-written as
\[\Ep{P^n}{\alpha_P(\traindata)} \leq \alpha,\]
where the expectation is with respect to the training data $\traindata \sim P^n$ 
(and, in order to be distribution-free, this bound is again required to hold for every distribution $P$ on $\Xcal\times\R$). 

While this expectation is bounded, $\alpha_P(\traindata)$ may have high variance over the training data. 
In particular, we can consider a worst-case scenario where
\begin{equation}\label{eqn:worstcase}
\Pp{P^n}{\alpha_P(\traindata)=1} = \alpha, \quad \Pp{P^n}{\alpha_P(\traindata)=0} =1- \alpha,
\end{equation}
which trivially satisfies the marginal coverage guarantee~\eqref{eqn:DF_marginal}
since $\Ep{P^n}{\alpha_P(\traindata)}=\alpha$.
In other words, in this worst-case scenario, a nonnegligible proportion of training sets might result in $0\%$ training-conditional coverage even though the average coverage is still $1-\alpha$, which may be highly problematic in practice.
On the other hand, if we instead had $\alpha_P(\traindata)\approx \alpha$ with high probability over $\traindata\sim P^n$,
this would be ideal, since it ensures that for nearly every possible draw of the training data, the resulting
coverage over future test points should be $\approx 1-\alpha$. 

The variability of training-conditional miscoverage level $\alpha_P(\traindata)$ will in general
depend on the distribution $P$, the regression algorithm $\alg$, and the particular distribution-free method that is used to generate $\Chn(X_{n+1})$. In this paper, we examine the variability of coverage for popular distribution-free methods for arbitrary $P$ and $\alg$. In particular, we seek to provide guarantees of the form
\begin{equation}
\label{eqn:bestcase_constant}\Pp{P^n}{\alpha_P(\traindata) > \alpha + \littleo(1)} \leq \littleo(1),
\end{equation}
also known as a ``Probably Approximately Correct'' (PAC) predictive interval.
This type of guarantee turns out to be possible for some distribution-free methods and impossible for others. In addition, we demonstrate empirically that methods without guarantees of this form can exhibit highly variable training-conditional miscoverage rates in low stability regimes.

\section{Background}
 
In this section, we will briefly review four related methods
for distribution-free predictive inference, to introduce the methods that we will study in this work.
 Consider an algorithm $\alg$ that maps datasets (consisting of $(X,Y)$ pairs, with features $X\in\Xcal$ and 
a real-valued response $Y\in\R$), to fitted regression functions $\muh:\Xcal\rightarrow\R$.

For a new data point whose features $X_{n+1}\in\Xcal$ are observed, 
we would like to predict the unseen response $Y_{n+1}\in\R$.
 Given a model $\muh$ obtained by training some algorithm $\alg$ on the available training data $(X_1,Y_1),\dots,(X_n,Y_n)$,
can we construct a prediction interval for $Y_{n+1}$ around the estimate $\muh(X_{n+1})$?
In  many practical
settings, the distribution of the data is likely unknown, and the regression algorithm $\alg$ may be a complex ``black
box'' methods whose theoretical properties are not well understood, and therefore it may be challenging
to guarantee a particular error bound for $\muh(X_{n+1})$ as an estimator of the unseen response $Y_{n+1}$.

\subsection{Distribution-free methods}
\subsubsection{Conformal prediction}
The conformal prediction framework  \citep{vovk2005algorithmic},
which includes the full and split conformal methods (also called ``transductive'' and ``inductive'' conformal, respectively),
 provides a mechanism for constructing prediction intervals in this challenging setting, 
 with distribution-free coverage guarantees. (See also \citet{lei2018distribution} for additional
 background on these methods.)
 
To run the split conformal method, we first partition the $n$ available labeled data points into 
a training set of size $n_0$ and a holdout set of size $n_1=n-n_0$. After running the regression
algorithm on the training data to obtain the fitted model $\muh_{n_0} = \alg\big((X_1,Y_1),\dots,(X_{n_0},Y_{n_0})\big)$,
the prediction interval is defined as
\begin{equation}\label{eqn:define_split_conformal}
\Chn(X_{n+1}) = \muh_{n_0}(X_{n+1})\pm \widehat{Q}_{n_1},\end{equation}
where
$\widehat{Q}_{n_1}$ is defined as the $\lceil(1-\alpha)(n_1+1)\rceil$-th smallest value of 
the holdout residuals
$|Y_{n_0+1} - \muh_{n_0}(X_{n_0+1})|, \dots , |Y_n-\muh_{n_0}(X_n)|$.
This method satisfies the marginal distribution-free predictive coverage guarantee~\eqref{eqn:DF_marginal} \citep{vovk2005algorithmic}.

While split conformal offers both computational efficiency and distribution-free coverage,
its precision may suffer from the loss of sample size incurred by splitting the data set.
In contrast, full conformal 
uses all the available training data for model fitting, but comes at a high computational cost.
Specifically, for every $y\in\R$, define $\muh_{n+1}^y = \alg\big((X_1,Y_1),\dots,(X_n,Y_n),(X_{n+1},y)\big)$,
the fitted model obtained by running algorithm $\alg$ on the training data together with the hypothesized
test point $(X_{n+1},y)$. Then construct the prediction set
\begin{equation}\label{eqn:define_full_conformal}\Chn(X_{n+1}) = \left\{y \in \R: |y - \muh^y_{n+1}(X_{n+1})| \leq \widehat{Q}_{n+1}^y\right\},\end{equation}
where $\widehat{Q}_{n+1}^y$ is defined as  the $\lceil(1-\alpha)(n+1)\rceil$-th smallest value of 
the residuals $|Y_1 - \muh^y_{n+1}(X_1)|,\dots,|Y_n - \muh^y_{n+1}(X_n)|,|y - \muh^y_{n+1}(X_{n+1})|$.
Full conformal prediction also offers the distribution-free coverage guarantee~\eqref{eqn:DF_marginal} \citep{vovk2005algorithmic},
 under one additional assumption---the 
algorithm $\alg$ needs to be {\em symmetric} in the training data points, meaning
that for any $m\geq 1$, any permutation $\sigma$ on $[m]:=\{1,\dots,m\}$, and any data points $(x_1,y_1),\dots,(x_m,y_m)\in\Xcal\times\R$,
\begin{equation}\label{eqn:alg_symmetric}\alg\big((x_1,y_1),\dots,(x_m,y_m)\big) = \alg\big((x_{\sigma(1)},y_{\sigma(1)}),\dots,(x_{\sigma(m)},y_{\sigma(m)})\big) .\end{equation}
Full conformal prediction is generally more statistically efficient than split conformal (i.e., will provide narrower prediction intervals)
since we do not need to split the training data. On the other hand, the computational cost is high---aside from special
cases (e.g., choosing $\alg$ to be the Lasso \citep{lei2019fast}), the prediction interval can only be calculated by running the regression algorithm 
$\alg$ for every possible $y\in\R$, or in practice, for a very fine grid of $y$ values (theoretical guarantees for this 
discretized setting can also be obtained, as shown by \citet{chen2018discretized}).
 
 \subsubsection{Jackknife+ and CV+}
 The jackknife+ and CV+ methods proposed by \citet{barber2021predictive}
 offer a compromise between the computational efficiency of split conformal and the statistical efficiency
 of full conformal. These methods, which are closely related to the cross-conformal procedure of \citet{vovk2015cross,vovk2018cross},
use a cross-validation type approach. 

For jackknife+, let $\muh_{[n]\backslash\{i\}}$ denote the model fitted to the training data with data point $i$ removed,
\[\muh_{[n]\backslash\{i\}} = \alg\big((X_1,Y_1),\dots,(X_{i-1},Y_{i-1}),(X_{i+1},Y_{i+1}),\dots,(X_n,Y_n)\big).\]
Then the jackknife+ prediction interval is defined as
\begin{multline}\label{eqn:define_jackknife+}
\Chn(X_{n+1}) = 
\bigg[\textnormal{the $\lceil(1-\alpha)(n+1)\rceil$-th largest of $\{\muh_{[n]\backslash \{i\}}(X_{n+1}) - R_i\}_{i\in[n]}$},\\
\textnormal{the $\lceil(1-\alpha)(n+1)\rceil$-th smallest of $\{\muh_{[n]\backslash \{i\}}(X_{n+1}) + R_i\}_{i\in[n]}$}\bigg],
\end{multline}
where $R_i = |Y_i - \muh_{[n]\backslash\{i\}}(X_i)|$ for $i\in[n]:=\{1,\dots,n\}$.
The jackknife+ method offers a weaker distribution-free coverage guarantee \citep[Theorem 1]{barber2021predictive}:
 for every distribution $P$ on $\Xcal\times\R$, assuming $\alg$ is symmetric as in~\eqref{eqn:alg_symmetric},
\begin{equation}\label{eqn:DF_jackknife+}\Pp{P^{n+1}}{Y_{n+1}\in\Chn(X_{n+1})} \geq 1-2\alpha.\end{equation}
Note that, in this theoretical guarantee, noncoverage may be as high as $2\alpha$, rather than the target level $\alpha$.
However, empirically the method typically achieves coverage at level $1-\alpha$,
and indeed, under algorithmic stability assumptions, e.g.,
\begin{equation}\label{eqn:assume_stability}
\Pp{P^{n+1}}{\left|\muh_n(X_{n+1}) - \muh_{[n]\backslash\{i\}}(X_{n+1})\right|\leq \eps}\geq 1-\nu,
\end{equation}
the predictive coverage guarantee can be improved to $1-\alpha - \littleo(1)$ \citep[Theorem 5]{barber2021predictive}.

While jackknife+ requires only $n$ many calls to the regression algorithm $\alg$ (in contrast to full conformal, which 
in theory requires infinitely many calls), for a large sample size $n$ this computational cost may still be too high.
CV+ extends the jackknife+ method to $K$-fold cross-validation (where we can view jackknife+ as $n$-fold cross-validation,
i.e., $K=n$). 
Let $[n] = S_1\cup\dots\cup S_K$ be a partition of the training data into $K$ subsets of size $n/K$, and write $\muh_{[n]\backslash S_k}$
as the fitted model when the $k$-th fold $S_k$ is removed from the $n$ training data points.
The CV+ prediction interval is defined as
\begin{multline}\label{eqn:define_CV+}
\Chn(X_{n+1}) = 
\bigg[\textnormal{the $\lceil(1-\alpha)(n+1)\rceil$-th largest of $\{\muh_{[n]\backslash S_{k(i)}}(X_{n+1}) - R_i\}_{i\in[n]}$},\\
\textnormal{the $\lceil(1-\alpha)(n+1)\rceil$-th smallest of $\{\muh_{[n]\backslash S_{k(i)}}(X_{n+1}) + R_i\}_{i\in[n]}$}\bigg],
\end{multline}
where now $R_i = |Y_i - \muh_{[n]\backslash S_{k(i)}}(X_i)|$ for $i\in[n]$, and where $k(i)$ denotes the fold
to which data point $i$ belongs, i.e., $i\in S_{k(i)}$.
The CV+ method's coverage guarantee is given by \citet[Theorem 4]{barber2021predictive} (see also \citet{vovk2018cross}
for a partial version of this result):
 for every distribution $P$ on $\Xcal\times\R$, assuming $\alg$ is symmetric as in~\eqref{eqn:alg_symmetric},
\begin{equation}\label{eqn:DF_CV+}\Pp{P^{n+1}}{Y_{n+1}\in\Chn(X_{n+1})} \geq 1-2\alpha - \sqrt{2/n}.\end{equation}
As for jackknife+,
the CV+ method typically achieves coverage near or above the target level $1-\alpha$ in practice.

\subsubsection{A note on randomized algorithms}
The background given above implicitly treats the algorithm $\alg$ as a {\em deterministic} function of the 
training data---that is, we view $\alg$ as a function $\big((X_1,Y_1),\dots,(X_n,Y_n)\big)\mapsto\muh$.
In many settings, however, it is common to use a {\em randomized} regression algorithm---for instance,
stochastic gradient descent. In this setting, we can formally view $\alg$ as a function 
$\big((X_1,Y_1),\dots,(X_n,Y_n),\xi\big)\mapsto\muh$, where the term $\xi$ introduces stochastic noise (effectively,
a random seed). All the results described above hold for both the deterministic and randomized settings.
(For results that assume $\alg$ is symmetric, the symmetry condition~\eqref{eqn:alg_symmetric} should be understood
in the distributional sense---that is, the training data points are treated symmetrically with respect to the randomized
training procedure. For example, for stochastic gradient descent, if data points are drawn uniformly at random
during the training epochs, then symmetry is satisfied.)

\subsection{Marginal or conditional validity}\label{sec:marginal_or_conditional}
The predictive coverage bound~\eqref{eqn:DF_marginal} achieved by split and full conformal,
or the bounds~\eqref{eqn:DF_jackknife+} and~\eqref{eqn:DF_CV+} for the jackknife+ and CV+ methods, 
are all {\em marginal} guarantees. This means that the probability is calculated over a random draw of both
the training and test data.
However, this may be unsatisfactory for practical purposes, in several ways.

\paragraph{Training-conditional coverage}
First, as discussed in Section~\ref{sec:intro} above, we may be interested in {\em training-conditional coverage},
which ensures that the predictive coverage guarantees hold (at least approximately)
even after conditioning on the training data set $\traindata = \big((X_1,Y_1),\dots,(X_n,Y_n)\big)$.
For the split conformal method described in~\eqref{eqn:define_split_conformal},
\citet[Proposition 2a]{vovk2012conditional} establishes training-conditional coverage through a Hoeffding bound:
\begin{theorem}[{\citet[Proposition 2a]{vovk2012conditional}}]\label{thm:split_conformal}
Consider the split conformal method defined in~\eqref{eqn:define_split_conformal}
with sample size $n=n_0+n_1$, where $n_0\geq 1$ many data points are used for training the fitted model $\muh_{n_0}$ 
(with an arbitrary
algorithm) while the remaining $n_1\geq1$ data points are used as the holdout set. Then, for any distribution $P$ and any $\delta\in(0,0.5]$,
\[\Pp{P^n}{\alpha_P(\traindata) \leq \alpha + \sqrt{\frac{\log(1/\delta)}{2n_1}}} \geq 1- \delta.\]
\end{theorem}
\noindent  This result holds for both deterministic and randomized algorithms $\alg$. Note that it is not necessary to assume that $\alg$ is symmetric.

In other words, the probability that a training set results in a significantly higher training-conditional miscoverage rate than the nominal rate, is vanishingly small under the split conformal method.
Of course, by running split conformal at a modified value $\alpha' := \alpha - \sqrt{\frac{\log(1/\delta)}{2n_1}}$,
we would obtain a slightly more conservative prediction interval that would then satisfy
\[\Pp{P^n}{\alpha_P(\traindata) \leq \alpha} \geq 1- \delta.\]
This type of guarantee (i.e., with probability at least $1-\delta$, we obtain at least $1-\alpha$ coverage, where $\alpha$ and $\delta$
are specified by the user) is often referred to as a {\em probably approximately correct} (PAC) guarantee.
This style of inference guarantee dates back to the work of \citet{wilks1941determination,wald1943extension} on setting
``tolerance limits'', i.e., a prediction interval (in the univariate case) or prediction region (in the multivariate case),
for a random variable $Y\sim P$, given $n$ i.i.d.~draws $Y_1,\dots,Y_n\iidsim P$ (that is, a prediction region $\Chn$ for $Y$ without any covariate $X$,
such that $P(\Chn)\geq 1-\alpha$ holds with probability at least $1-\delta$, for user-specified parameters $\alpha$ and $\delta$). 
More recent results offering PAC-style
training-conditional coverage guarantees for the regression setting,
via split conformal and related methods, can be found in the work of \citet{kivaranovic2020adaptive,bates2021distribution,yang2021finite,park2020pac};
see also \citet{park2021pac,qiu2022distribution,yang2022doubly} for training-conditional coverage under covariate shift.

No analogous finite-sample results are known for
distribution-free prediction methods beyond split conformal,
although \citet{steinberger2018conditional}
 analyze asymptotic training-conditional validity for the jackknife and for cross-validation under algorithmic stability
type assumptions such as~\eqref{eqn:assume_stability}. In this work, our goal will be to examine the finite-sample
training-conditional coverage properties of distribution-free methods beyond split conformal.

\paragraph{Object-conditional or label-conditional coverage}
As a second way in which marginal coverage may not be sufficient for practical utility,
we may also be interested in coverage at a particular new test feature vector $X_{n+1}$ (referred
to in \citet{vovk2012conditional} as {\em object-conditional coverage})---for instance, if the data points correspond
to individual patients in a clinical setting, is it true
that a given patient with a particular feature vector $X_{n+1}=x$ has a $1-\alpha$ probability of a correct predictive interval?
That is, we would like to show that the conditional coverage probability $\Ppst{P^{n+1}}{Y_{n+1}\in\Chn(X_{n+1})}{X_{n+1}=x}$
is $\geq 1-\alpha$, at least approximately. However, \citet{vovk2012conditional,lei2014distribution}
show that this type of guarantee is impossible under any distribution $P$ for which $X$ is nonatomic
(i.e., $\Pp{P}{X=x}=0$ for all $x\in\Xcal$---for instance, this is satisfied by any continuous distribution on $\R^d$);
see also \citet{barber2021limits}. 
A third type of conditional guarantee is that of {\em label-conditional coverage} \citep{vovk2012conditional,lofstrom2015bias}
for the setting where the response $Y$ is categorical, requiring
accuracy conditional on the class, i.e., $\Ppst{P^{n+1}}{Y_{n+1}\in\Chn(X_{n+1})}{Y_{n+1}=y} \geq 1-\alpha$
for each category $y$.
Both of these type of conditional guarantees are fundamentally very different from training-conditional
coverage, and we will not address these further in this work.

\section{Theoretical results}
As shown in Theorem~\ref{thm:split_conformal} above, a training-conditional guarantee of the form~\eqref{eqn:bestcase_constant}
was established by \citet{vovk2012conditional} for the split conformal method.
In our work,
we find that a guarantee of the form~\eqref{eqn:bestcase_constant} can also be shown
 for the $K$-fold CV+ method (as long as $n/K$, the number of data points in each fold,
 is sufficiently large), but no such guarantees are possible for the full conformal or jackknife+ methods. In this section, we present the main results
for each of the three previously unstudied methods. The proofs will be given in Section~\ref{sec:proofs} below.

First, we consider the full conformal prediction method. In contrast to split conformal, 
 it is impossible to guarantee training-conditional coverage
for the full conformal method without further assumptions. 
\begin{theorem}\label{thm:full_conformal} For any sample size $n\geq 2$ and any distribution $P$ for which the marginal $P_X$ is nonatomic,
	 there exists a symmetric and deterministic regression algorithm $\alg$
	such that the full conformal prediction method defined in~\eqref{eqn:define_full_conformal}
	satisfies
	\[\Pp{P^n}{\alpha_P(\traindata) \geq 1-n^{-2}} \geq \alpha - 6\sqrt{\frac{\log n}{n}}.\]
\end{theorem}
\noindent In other words, without placing assumptions on the distribution
$P$ and/or the algorithm $\alg$  (beyond the standard symmetry assumption), we cannot avoid the worst-case scenario~\eqref{eqn:worstcase}, where the marginal 
guarantee of $1-\alpha$ coverage stated in~\eqref{eqn:DF_marginal} is achieved only because
the training data set yields $\approx 100\%$ coverage with probability $\approx 1-\alpha$, and $\approx 0\%$ coverage
with probability $\approx \alpha$.
(Our result holds only for distributions $P$
where $X$ is nonatomic, i.e., $\Pp{P}{X=x}=0$ for all $x\in\Xcal$---this condition appears also in the impossibility 
results for object-conditional coverage as described earlier in Section~\ref{sec:marginal_or_conditional}.)

Next, for the jackknife+, the same worst-case result holds.
\begin{theorem}\label{thm:jackknife+}  For any sample size $n\geq 2$ and any distribution $P$ for which the marginal $P_X$ is nonatomic, 
	 there exists a symmetric and deterministic regression algorithm $\alg$
	such that the jackknife+ prediction interval defined in~\eqref{eqn:define_jackknife+} satisfies
	\[\Pp{P^n}{\alpha_P(\traindata) \geq 1-n^{-2}} \geq \alpha - 6\sqrt{\frac{\log n}{n}}.\]
\end{theorem}
\noindent Thus, as for full conformal, without placing assumptions on $P$ and/or $\alg$   (beyond the standard symmetry assumption),
we cannot ensure that the jackknife+ method will avoid the worst-case
scenario~\eqref{eqn:worstcase}.

In contrast, for CV+, we will now see that the lower bound on marginal coverage, which is $\gtrapprox 1-2\alpha$
as shown in~\eqref{eqn:DF_CV+}, can also be obtained as a training-conditional guarantee.
\begin{theorem}\label{thm:CV+}
	For any integers $K\geq 2$ and $m\geq 1$, and let $n=Km$. 
	Suppose CV+ is run with $K$ folds each of size $m$. Then, 
	for any regression algorithm $\alg$ and any distribution $P$, the $K$-fold CV+ method~\eqref{eqn:define_CV+} satisfies
	\[\Pp{P^n}{\alpha_P(\traindata) \leq 2\alpha +  \sqrt{ \frac{2\log(K/\delta)}{m} }} \geq 1-  \delta\]
	for any $\delta > 0$.
\end{theorem}
\noindent As long as the size of each fold, $m=n/K$, is large, the bound on $\alpha_P(\traindata)$ is approximately 
$2\alpha$. Comparing to the marginal result~\eqref{eqn:DF_CV+} for the CV+ method,
we see that the conditional coverage guarantee (for ``most'' training data sets $\traindata$) essentially matches the marginal coverage 
guarantee, and thus could not be improved.  Note also that, as in Theorem~\ref{thm:split_conformal} for split conformal,
we do not need to assume $\alg$ is symmetric, and the result holds regardless of whether $\alg$ is deterministic or randomized.

\section{Proofs}\label{sec:proofs} 
Before proceeding to the proofs, we give some brief intuition for why the split conformal and CV+ methods offer training-conditional
coverage guarantees, while full conformal and jackknife+ do not. For split conformal, the $n_1$ residuals on the holdout set
\[\left\{|Y_i - \muh_{n_0}(X_i)| : i = n_0+1,\dots,n\right\}\]
are i.i.d.~after conditioning on the fitted model $\muh_{n_0}$, and therefore, for large $n_1$, their sample quantiles concentrate
around the corresponding population quantiles. Similarly, for CV+, for each fold $k=1,\dots,K$ we have $m=n/K$ many residuals
\[\left\{|Y_i - \muh_{[n]\backslash S_k}(X_i)|: i \in S_k\right\}\]
that are again i.i.d.~conditional on the $k$-th fitted model $\muh_{[n]\backslash S_k}$, and thus again their 
sample quantiles concentrate
as long as the fold size $m=n/K$ 
is large. This concentration of the sample quantiles (which we formalize in Lemma~\ref{lem:holdout} below)
is the key ingredient for establishing training-conditional coverage. On the other hand,
for both full conformal and jackknife+, there is no independence among residuals calculated in each method---for
example, for jackknife+, in the leave-one-out residuals $R_i = |Y_i - \muh_{[n]\backslash\{i\}}(X_i)|$, 
for two data points $i\neq j$, data point $(X_i,Y_i)$ is used for training when computing the $j$-th residual $R_j$,
and vice versa.

\subsection{Proofs for split conformal and CV+}
We begin by considering the split conformal and CV+ methods,
which both achieve training-conditional coverage.
Both the split conformal result, Theorem~\ref{thm:split_conformal} {\citep[Proposition 2a]{vovk2012conditional}},
and the CV+ result, Theorem~\ref{thm:CV+}, can be proved as consequences of the following lemma.
\begin{lemma}\label{lem:holdout}
Let $n\geq 2$ and choose a holdout set $A$ with $\emptyset\subsetneq A \subsetneq [n]$. Let $\muh_{[n]\backslash A}= \alg\big((X_i,Y_i): i\in [n]\backslash A\big)$,  where $\alg$ is any algorithm and may be deterministic or randomized.
Define
\[p_A(x,y) = \frac{1}{|A|}\sum_{i \in A} \One{|Y_i - \muh_{[n]\backslash A}(X_i)| \geq |y - \muh_{[n]\backslash A}(x)|  },\]
and
\[p^*_A(x,y) = \Ppst{P}{ |Y - \muh_{[n]\backslash A}(X)| \geq |y - \muh_{[n]\backslash A}(x)|  }{\muh_{[n]\backslash A}}.\]
Then $p^*_A(X_{n+1},Y_{n+1})$ is a valid p-value conditional on the training data, i.e.,
\begin{equation}\label{eqn:lemma_pvalue}\Ppst{P}{p^*_A(X_{n+1},Y_{n+1})\leq a}{\traindata}\leq a\textnormal{ for all $a\in[0,1]$},\textnormal{ almost surely over $\traindata$}.\end{equation}
Moreover, for any $\Delta \geq \sqrt{\frac{\log 2}{2|A|}}$,
\begin{equation}\label{eqn:lemma_DKW}\Pp{P^n}{\sup_{(x,y)\in\Xcal\times \R} \left(p^*_A(x,y) - p_A(x,y)\right) >\Delta}\leq e^{-2|A|\Delta^2},\end{equation}
\end{lemma}

Next we will see how this lemma implies the two theorems. First, for split conformal,
Theorem~\ref{thm:split_conformal} is proved by \citet[Proposition 2a]{vovk2012conditional}, but here we reformulate the proof in terms of the above lemma,
to set up intuition for our CV+ proof later on.
\begin{proof}[Proof of Theorem~\ref{thm:split_conformal} {\citep[Proposition 2a]{vovk2012conditional}}]
By definition of the split conformal method~\eqref{eqn:define_split_conformal}, we have
\begin{align*}
\Chn(X_{n+1}) 
&=\left\{y\in\R : \widehat{Q}_{n_1} \geq |y - \muh_{n_0}(X_{n+1})|\right\}\\
&=\left\{y\in\R : \sum_{i = n_0+1}^n \One{|Y_i - \muh_{n_0}(X_i)| <  |y - \muh_{n_0}(X_{n+1})|  } < \lceil(1-\alpha)(n_1+1)\rceil\right\}\\
&\supseteq \left\{y\in\R : \sum_{i = n_0+1}^n \One{ |Y_i - \muh_{n_0}(X_i)| \geq |y - \muh_{n_0}(X_{n+1})|  } >  \alpha n_1\right\}\\
&=  \left\{y\in\R :  p_{[n]\backslash [n_0]}(X_{n+1},y) > \alpha \right\}.
\end{align*}
where $p_{[n]\backslash [n_0]}(x,y)$ is defined as in Lemma~\ref{lem:holdout} by choosing the holdout set $A=[n]\backslash [n_0]$.
Therefore,  we have
\[\alpha_P(\traindata) = \Ppst{P}{Y_{n+1} \not\in \Chn(X_{n+1})}{\traindata}
\leq \Ppst{P}{ p_{[n]\backslash [n_0]}(X_{n+1},Y_{n+1}) \leq \alpha}{\traindata}.\]
Next, fixing any $\Delta>0$, consider the event that $\sup_{(x,y)\in\Xcal\times \R} \left(p^*_{[n]\backslash [n_0]}(x,y) - p_{[n]\backslash [n_0]}(x,y)\right) \leq \Delta$,
which is a function of the training data $\traindata$. On this event, we have
\[\alpha_P(\traindata) \leq \Ppst{P}{ p^*_{[n]\backslash [n_0]}(X_{n+1},Y_{n+1}) \leq \alpha + \Delta}{\traindata} \leq \alpha+\Delta,\]
where the last step holds by~\eqref{eqn:lemma_pvalue}. In other words, so far we have shown that
\[\Pp{P^n}{\alpha_P(\traindata)>\alpha+\Delta}\leq \Pp{P^n}{\sup_{(x,y)\in\Xcal\times \R} \left(p^*_{[n]\backslash [n_0]}(x,y) - p_{[n]\backslash [n_0]}(x,y)\right)>\Delta}.\]
Finally, applying~\eqref{eqn:lemma_DKW}, this probability is bounded by $\delta$ when we choose $\Delta =\sqrt{\frac{\log(1/\delta)}{2n_1}}$.
\end{proof}

Next, we prove Theorem~\ref{thm:CV+} for the CV+ method.
\begin{proof}[Proof of Theorem~\ref{thm:CV+}]
As in the definition of the CV+ method~\eqref{eqn:define_CV+},
we let $R_i = |Y_i - \muh_{[n]\backslash S_{k(i)}}(X_i)|$ for each $i\in[n]$. 
Following \citet[Proof of Theorem 4]{barber2021predictive},
  the CV+ prediction interval defined in~\eqref{eqn:define_CV+} deterministically satisfies
\begin{align*}
\Chn(X_{n+1})
& \supseteq \left\{y\in\R : \sum_{i=1}^n \One{|y -  \muh_{[n]\backslash S_{k(i)}}(X_{n+1})| > R_i} < (1-\alpha)(n+1)\right\}\\
& \supseteq \left\{y\in\R : \sum_{i=1}^n \One{R_i \geq |y -  \muh_{[n]\backslash S_{k(i)}}(X_{n+1})| } > \alpha n\right\}\\
&=  \left\{y\in\R :  \sum_{k=1}^K \sum_{i\in S_k} \One{R_i \geq |y -  \muh_{[n]\backslash S_k}(X_{n+1})| } > \alpha n\right\}\\
&=  \left\{y\in\R :  \frac{1}{K}\sum_{k=1}^K p_{S_k}(X_{n+1},y) > \alpha \right\},
\end{align*}
where for each fold $k$, $p_{S_k}(x,y)$ is defined as in Lemma~\ref{lem:holdout} by choosing the holdout set $A= S_k$.
Therefore,  we have
\[\alpha_P(\traindata) = \Ppst{P}{Y_{n+1} \not\in \Chn(X_{n+1})}{\traindata}
\leq \Ppst{P}{\frac{1}{K}\sum_{k=1}^K p_{S_k}(X_{n+1},Y_{n+1}) \leq \alpha}{\traindata}.\]
Next, fixing any $\Delta>0$, consider the event that
 $\max_k \sup_{(x,y)\in\Xcal\times \R} \left(p^*_{S_k}(x,y) - p_{S_k}(x,y)\right) \leq \Delta$, 
 which is a function of the training data $\traindata$. On this event, we have
\[\alpha_P(\traindata) \leq \Ppst{P}{ \frac{1}{K}\sum_{k=1}^Kp^*_{S_k}(X_{n+1},Y_{n+1}) \leq \alpha + \Delta}{\traindata} \leq 2\alpha+2\Delta,\]
where the last step holds 
because each $p^*_{S_k}(X_{n+1},Y_{n+1})$ is a valid p-value conditional on $\traindata$ by~\eqref{eqn:lemma_pvalue},
and the average of valid p-values 
is itself a p-value up to a factor of 2 \citep{ruschendorf1982random,vovk2020combining}. 
Combining everything so far, we have shown that
\begin{multline*}\Pp{P^n}{\alpha_P(\traindata)>2\alpha+2\Delta}\leq \Pp{P^n}{\max_k\sup_{(x,y)\in\Xcal\times \R} \left(p^*_{S_k}(x,y) - p_{S_k}(x,y)\right)>\Delta}\\
\leq \sum_{k=1}^K \Pp{P^n}{\sup_{(x,y)\in\Xcal\times \R} \left(p^*_{S_k}(x,y) - p_{S_k}(x,y)\right)>\Delta},\end{multline*}
where for the last step we take a union bound.
Finally, applying~\eqref{eqn:lemma_DKW} to bound this probability for each fold $k$, the above quantity
 is bounded by $\delta$ when we choose $\Delta =\sqrt{\frac{\log(K/\delta)}{2m}}$.
\end{proof}

To conclude this section, we now prove the supporting lemma.
\begin{proof}[Proof of Lemma~\ref{lem:holdout}]
First, for any fixed function $\mu:\Xcal\rightarrow\R$, define
\[\bar{F}_\mu(t) = \Pp{P}{|Y-\mu(X)|\geq t}.\]
In other words, $\bar{F}_\mu$ is right-tailed CDF of $|Y-\mu(X)|$ under $(X,Y)\sim P$.
We can therefore write
\[p^*_A(X_{n+1},Y_{n+1}) = \bar{F}_{\muh_{[n]\backslash A}}(|Y_{n+1} - \muh_{[n]\backslash A}(X_{n+1})|).\]
Since $(X_{n+1},Y_{n+1})\sim P$ (and is independent of $\muh_{[n]\backslash A}$), this is clearly a valid p-value by definition of $\bar{F}_{\muh_{[n]\backslash A}}$, and so we have proved~\eqref{eqn:lemma_pvalue}.

Next, for any $(x,y)\in\Xcal\times\R$, we can calculate
\begin{multline*}
p^*_A(x,y) - p_A(x,y)\\
=\bar{F}_{\muh_{[n]\backslash A}}(|y - \muh_{[n]\backslash A}(x)|) - \frac{1}{|A|}\sum_{i \in A} \One{|Y_i - \muh_{[n]\backslash A}(X_i)| \geq |y - \muh_{[n]\backslash A}(x)|  }\\
\leq\sup_{t\in\R}\left(\bar{F}_{\muh_{[n]\backslash A}}(t) - \frac{1}{|A|}\sum_{i\in A}\One{|Y_i - \muh_{[n]\backslash A}(X_i)| \geq t}\right) .\end{multline*}
Finally, since $\{(X_i,Y_i)\}_{i\in A}$ are drawn i.i.d.~from $P$ and are independent from $\muh_{[n]\backslash A}$,
for any $\Delta \geq \sqrt{\frac{\log 2}{2|A|}}$
the Dvoretzky--Kiefer--Wolfowitz inequality implies that, conditional on $\muh_{[n]\backslash A}$,
\[\sup_{t\in\R}\left( \bar{F}_{\muh_{[n]\backslash A}}(t) - \frac{1}{|A|}\sum_{i\in A}\One{|Y_i - \muh_{[n]\backslash A}(X_i)| \geq t}\right) \leq  \Delta\]
holds with probability at least $1-e^{-2|A|\Delta^2}$. The same bound therefore holds
marginally as well. This proves~\eqref{eqn:lemma_DKW}.
\end{proof}

\subsection{Proofs for full conformal and jackknife+}
Next, we turn to the results for full conformal and for jackknife+, where we show that
training-conditional coverage cannot be guaranteed without further assumptions.
The proofs for the two methods are closely related and share the same structure.

First, fix some large integer $M$ (which we will specify later), and partition $\Xcal$ into $M$ sets, $\Xcal = A_0 \cup A_1 \cup \dots \cup A_{M-1}$, 
where $\Pp{P}{X\in A_m} = \frac{1}{M}$ for each $m=0,\dots,M-1$ (since we have assumed $X$ is nonatomic
under the distribution $P$, such a partition exists by \citet[Proposition A.1]{dudley2011concrete}). 
Define  a map $a:\Xcal\rightarrow\{0,\dots,M-1\}$,
\[a(x) = \begin{cases}0,&x\in A_0,\\1 , & x \in A_1,\\\dots \\ M-1, &x\in A_{M-1},\end{cases}\]
assigning each $x\in\Xcal$ to a particular set in the partition. Then, by our choice of the $A_m$'s, we see that
\[a(X)\sim\textnormal{Unif}\{0,\dots,M-1\}\]
under the distribution $P$. By extension, $\mod(\sum_{i=1}^{n} a(X_i),M) \sim \textnormal{Unif}\{0,...,M-1\}$. For the sake of illustration, consider a ``clock'' partitioned into $M$ segments and a hand whose position represents the value of the modulo. In this case, the hand moves forward by $a(X_i)$ segments when the $i$-th term is added inside the modulo---see Figure~\ref{fig:clock1} for an illustration.

\begin{figure}[t]\centering
 \begin{tikzpicture}[line cap=rect,line width=1pt, scale=0.9]
    \filldraw [fill=white] (0,0) circle [radius=2.5cm];
    \foreach \angle in {60,30,0,-30,-60,-90,-120,-210,-240,-270}
    {
      \draw[line width=0.8pt] (0,0) -- (\angle:2.5cm);
    }
    \node at (75 : 2cm) {\textnormal{0}};
    \node at (45 : 2cm) {\textnormal{1}};
    \node at (15 : 2cm) {\textnormal{2}};
    \node at (-15 : 2cm) {\textnormal{3}};
    \node at (-45 : 2cm) {\textnormal{4}};
    \node at (-75 : 2cm) {\textnormal{5}};
    \node at (-105 : 2cm) {\textnormal{6}};
    \node at (105 : 2cm) {\textnormal{$M$-1}};
    \node at (135 : 2cm) {\textnormal{$M$-2}};
    \foreach \angle in {180,195, 210}
    {
        \node at (\angle : 2cm) [circle,fill,inner sep=0.8pt]{};
    }
    \draw[very thick,color=red] (-45:-1cm) -- (-45:2.4cm);
    \draw[line width=0.1cm,color=red] (-45:-1cm) -- (-45:-0.3cm);
    \draw[ ->] (0.258819*2.7, 0.9659258*2.7) arc (75:-45:2.7cm);
    \draw[fill=black] (0,0) circle (0.04cm);
    \draw [color=white] (0,0) circle [radius=2.9cm];
    \end{tikzpicture}\hspace{-.15in}
    \begin{tikzpicture}[line cap=rect,line width=1pt,scale=0.9]
    \filldraw [fill=white] (0,0) circle [radius=2.5cm];
    \foreach \angle in {60,30,0,-30,-60,-90,-120,-210,-240,-270}
    {
      \draw[line width=0.8pt] (0,0) -- (\angle:2.5cm);
    }
    \node at (75 : 2cm) {\textnormal{0}};
    \node at (45 : 2cm) {\textnormal{1}};
    \node at (15 : 2cm) {\textnormal{2}};
    \node at (-15 : 2cm) {\textnormal{3}};
    \node at (-45 : 2cm) {\textnormal{4}};
    \node at (-75 : 2cm) {\textnormal{5}};
    \node at (-105 : 2cm) {\textnormal{6}};
    \node at (105 : 2cm) {\textnormal{$M$-1}};
    \node at (135 : 2cm) {\textnormal{$M$-2}};
    \foreach \angle in {180,195, 210}
    {
        \node at (\angle : 2cm) [circle,fill,inner sep=0.8pt]{};
    }
    \draw[very thick,color=red] (45:-1cm) -- (45:2.4cm);
    \draw[line width=0.1cm,color=red] (45:-1cm) -- (45:-0.3cm);
    
    \draw[ ->] (0.7071068*2.7, -0.7071068*2.7) arc (-45 : -315 : 2.7cm);
    \draw[fill=black] (0,0) circle (0.04cm);
    \draw [color=white] (0,0) circle [radius=2.9cm];
    \end{tikzpicture}\hspace{-.2in}
    \begin{tikzpicture}[line cap=rect,line width=1pt,scale=0.9]
    \filldraw [fill=white] (0,0) circle [radius=2.5cm];
    \foreach \angle in {60,30,0,-30,-60,-90,-120,-210,-240,-270}
    {
      \draw[line width=0.8pt] (0,0) -- (\angle:2.5cm);
    }
    \node at (75 : 2cm) {\textnormal{0}};
    \node at (45 : 2cm) {\textnormal{1}};
    \node at (15 : 2cm) {\textnormal{2}};
    \node at (-15 : 2cm) {\textnormal{3}};
    \node at (-45 : 2cm) {\textnormal{4}};
    \node at (-75 : 2cm) {\textnormal{5}};
    \node at (-105 : 2cm) {\textnormal{6}};
    \node at (105 : 2cm) {\textnormal{$M$-1}};
    \node at (135 : 2cm) {\textnormal{$M$-2}};
    \foreach \angle in {180,195, 210}
    {
        \node at (\angle : 2cm) [circle,fill,inner sep=0.8pt]{};
    }
    \draw[very thick,color=red] (15:-1cm) -- (15:2.4cm);
    \draw[line width=0.1cm,color=red] (15:-1cm) -- (15:-0.3cm);
    
    \draw[ ->] (0.7071068*2.7, 0.7071068*2.7) arc (45:15:2.7cm);
    \draw[fill=black] (0,0) circle (0.04cm);
    \draw [color=white] (0,0) circle [radius=2.9cm];
    \end{tikzpicture}
	\caption{Representation of $\mod(\sum_{i=1}^{n} a(X_i),M)$ for $n=3$, 
	when $a(X_1) = 4$, $a(X_2) = M-3$, and $a(X_3) = 1$. The left plot shows moving from $0$ to $\mod(a(X_1),M)$,
	the middle plot shows moving from $\mod(a(X_1),M)$ to $\mod(a(X_1) + a(X_2),M)$,
	and the right plot shows moving from $\mod(a(X_1) + a(X_2),M)$ to $\mod(a(X_1) + a(X_2) + a(X_3),M)$.}
	\label{fig:clock1}
\end{figure}
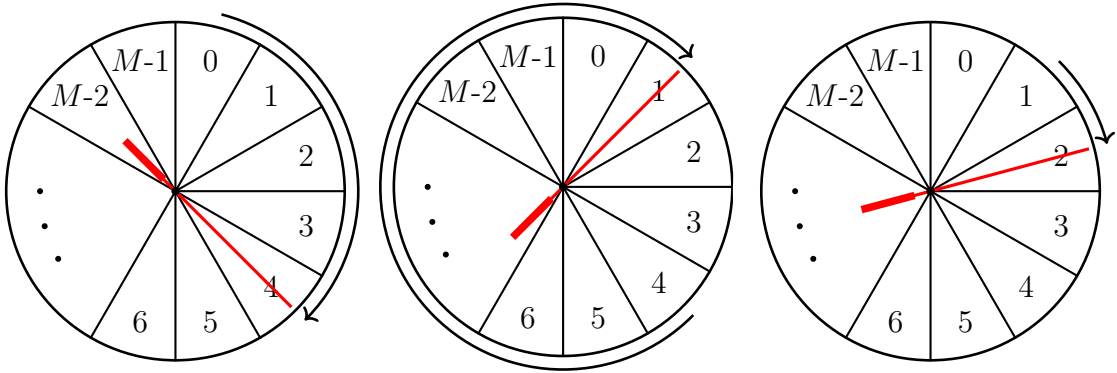

Next, we will define three events, $\Ecal_{\textnormal{max}}$, $\Ecal_{\textnormal{mod}}$, and $\Ecal_{\textnormal{unif}}$,
which are all functions of the training data $\traindata$.
Define
 \[y_* = \textnormal{ the $(1-n^{-2})$-quantile of $|Y|$ under distribution $P$},\]
and let
\[M_1 =\left\lfloor M\left(\alpha - \sqrt{\frac{2\log n}{n}} - \frac{2}{n}\right)\right\rfloor \approx \alpha M.\]
(Note that we can assume $n$ is sufficiently
large so that $\alpha - \sqrt{\frac{2\log n}{n}} - \frac{2}{n} >0$, and thus $M_1\geq 0$, since if this does not hold
then the results of the theorems hold trivially.)
With these values fixed, the three events are defined as follows:
\begin{itemize}
\item Let  $\Ecal_{\textnormal{max}}$ be the event that $\max_{i\in[n]}|Y_i| < y_*$.
\item Let $\Ecal_{\textnormal{mod}}$ be the event that $\mod(\sum_{i=1}^n a(X_i),M) < M_1$.
\item Let  $\Ecal_{\textnormal{unif}}$ be the event that 
 $\sum_{i=1}^n \One{\mod(a(X_i) + m,M) < M - M_1} \geq  \lceil (1-\alpha)(n+1)\rceil$
	for all integers $m$.
\end{itemize}
Figures~\ref{fig:clock2} and~\ref{fig:clock3} illustrate the events  $\Ecal_{\textnormal{mod}}$ and  $\Ecal_{\textnormal{unif}}$, 
respectively.

\begin{figure}[t]\centering
 \begin{tikzpicture}[line cap=rect,line width=1pt, scale=0.9]
    \filldraw [fill=white] (0,0) circle [radius=2.5cm];
 \draw[fill=cyan!20] (0,0)-- +(-30:2.5cm) arc (-30:90:2.5cm) -- cycle;
    \foreach \angle in {60,30,0,-30,-60,-90,-210,-240,-270}
    {
      \draw[line width=0.8pt] (0,0) -- (\angle:2.5cm);
    }
    \node at (75 : 2cm) {\textnormal{0}};
    \node at (45 : 2cm) {\textnormal{1}};
    \node at (-15 : 1.9cm) {\textnormal{\footnotesize$M_1$-1}};
    \node at (-45 : 2cm) {\textnormal{\footnotesize$M_1$}};
    \node at (-75 : 2.05cm) {\textnormal{\footnotesize$M_1$+1}};
    \node at (105 : 2cm) {\textnormal{$M$-1}};
    \node at (135 : 2cm) {\textnormal{$M$-2}};
    \foreach \angle in {195, 210, 225, 5, 15, 25}
    {
        \node at (\angle : 2cm) [circle,fill,inner sep=0.8pt]{};
    }
    \draw[fill=black] (0,0) circle (0.04cm);
        \draw [color=white] (0,0) circle [radius=3cm];
    \end{tikzpicture}
	\caption{An illustration of the event $\Ecal_{\textnormal{mod}}$,
	which is the event that $\mod(\sum_{i=1}^n a(X_i), M) < M_1$ for $M_1\approx \alpha M$.
	In the figure, the event $\Ecal_{\textnormal{mod}}$
	holds if and only if the value $\mod(\sum_{i=1}^n a(X_i), M)$ lands in the shaded region of
	the ``clock''.}
	\label{fig:clock2}
\end{figure}

\begin{figure}[t]\centering
 \begin{tikzpicture}[line cap=rect,line width=1pt, scale=0.9]
    \filldraw [fill=white] (0,0) circle [radius=2.5cm];
 \draw[fill=cyan!20] (0,0)-- +(-150:2.5cm) arc (-150:90:2.5cm) -- cycle;
    \foreach \angle in {60,30,-60,-90,-120,-150,-180,-210,-240,-270}
    {
      \draw[line width=0.8pt] (0,0) -- (\angle:2.5cm);
    }
    \node at (75 : 2cm) {\textnormal{0}};
    \node at (45 : 2cm) {\textnormal{1}};
    \node at (-75 : 2.15cm) {\textnormal{\scriptsize$M$\!\!\! - \!\!$M_1$\!-3}};
        \node at (-106 : 2.15cm) {\textnormal{\scriptsize$M$\!\!\! - \!\!$M_1$\!-\!2}};
        \node at (-138 : 2cm) {\textnormal{\scriptsize$M$\!\! - \!\!$M_1$\!-\!-1}};
       \node at (-165 : 1.95cm) {\textnormal{\scriptsize$M$\!\! - \!\!$M_1$}};
    \node at (105 : 2cm) {\textnormal{$M$-1}};
        \node at (135 : 2cm) {\textnormal{$M$-2}};
    \foreach \angle in {155, 165, 175,  0, -15, -30}
    {
        \node at (\angle : 2cm) [circle,fill,inner sep=0.8pt]{};
    }
    \draw[fill=black] (0,0) circle (0.04cm);
        \draw [color=white] (0,0) circle [radius=2.9cm];
    \end{tikzpicture}\hspace{-.15in}
     \begin{tikzpicture}[line cap=rect,line width=1pt, scale=0.9]
    \filldraw [fill=white] (0,0) circle [radius=2.5cm];
 \draw[fill=cyan!20] (0,0)-- +(-120:2.5cm) arc (-120:120:2.5cm) -- cycle;
    \foreach \angle in {60,30,-60,-90,-120,-150,-180,-210,-240,-270}
    {
      \draw[line width=0.8pt] (0,0) -- (\angle:2.5cm);
    }
    \node at (75 : 2cm) {\textnormal{0}};
    \node at (45 : 2cm) {\textnormal{1}};
    \node at (-75 : 2.15cm) {\textnormal{\scriptsize$M$\!\!\! - \!\!$M_1$\!-3}};
        \node at (-106 : 2.15cm) {\textnormal{\scriptsize$M$\!\!\! - \!\!$M_1$\!-\!2}};
        \node at (-138 : 2cm) {\textnormal{\scriptsize$M$\!\! - \!\!$M_1$\!-\!-1}};
       \node at (-165 : 1.95cm) {\textnormal{\scriptsize$M$\!\! - \!\!$M_1$}};
    \node at (105 : 2cm) {\textnormal{$M$-1}};
        \node at (135 : 2cm) {\textnormal{$M$-2}};
    \foreach \angle in {155, 165, 175,  0, -15, -30}
    {
        \node at (\angle : 2cm) [circle,fill,inner sep=0.8pt]{};
    }
    \draw[fill=black] (0,0) circle (0.04cm);
        \draw [color=white] (0,0) circle [radius=2.9cm];
    \end{tikzpicture}\hspace{-.15in}
     \begin{tikzpicture}[line cap=rect,line width=1pt, scale=0.9]
    \filldraw [fill=white] (0,0) circle [radius=2.5cm];
 \draw[fill=cyan!20] (0,0)-- +(-90:2.5cm) arc (-90:150:2.5cm) -- cycle;
    \foreach \angle in {60,30,-60,-90,-120,-150,-180,-210,-240,-270}
    {
      \draw[line width=0.8pt] (0,0) -- (\angle:2.5cm);
    }
    \node at (75 : 2cm) {\textnormal{0}};
    \node at (45 : 2cm) {\textnormal{1}};
    \node at (-75 : 2.15cm) {\textnormal{\scriptsize$M$\!\!\! - \!\!$M_1$\!-3}};
        \node at (-106 : 2.15cm) {\textnormal{\scriptsize$M$\!\!\! - \!\!$M_1$\!-\!2}};
        \node at (-138 : 2cm) {\textnormal{\scriptsize$M$\!\! - \!\!$M_1$\!-\!-1}};
       \node at (-165 : 1.95cm) {\textnormal{\scriptsize$M$\!\! - \!\!$M_1$}};
    \node at (105 : 2cm) {\textnormal{$M$-1}};
        \node at (135 : 2cm) {\textnormal{$M$-2}};
    \foreach \angle in {155, 165, 175,  0, -15, -30}
    {
        \node at (\angle : 2cm) [circle,fill,inner sep=0.8pt]{};
    }
    \draw[fill=black] (0,0) circle (0.04cm);
        \draw [color=white] (0,0) circle [radius=2.9cm];
    \end{tikzpicture}
\caption{An illustration of the event $\Ecal_{\textnormal{unif}}$,
	which is the event that $\mod(a(X_i) + m, M) < M - M_1$ holds for at least $\lceil (1-\alpha)(n+1)\rceil$ many 
	training data points $i$, for every integer $m$.
	In the figure, the shaded region $\{a\in \{0,\dots,M-1\} : \mod(a + m, M) < M - M_1\}$ is shown for $m=0$ (left figure), $m=1$ (center figure), and $m=2$ (right figure). 
	The event holds if and only if at least  $\lceil (1-\alpha)(n+1)\rceil$ many 
	training data points $i\in[n]$ have $a(X_i)$ lying in the shaded region, for each integer $m$ (i.e., for each of the three displayed
	figures, as well as all other possible values of $m$).
}
	\label{fig:clock3}
\end{figure}
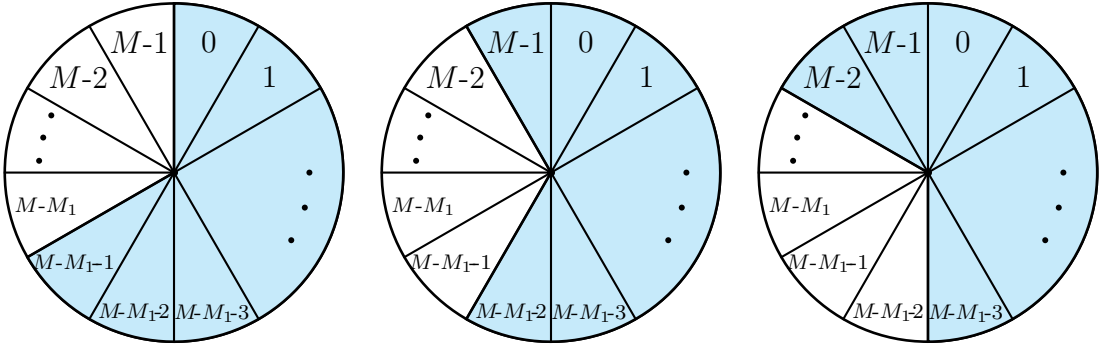

The following result shows that, with probability at least $\approx \alpha$, all three events occur:
\begin{lemma}\label{lem:three_events}
Under the definitions and notation above, for $(X_1,Y_1),\dots,(X_n,Y_n)\iidsim P$, we have
\[\Pp{P^n}{\Ecal_{\textnormal{mod}}} \geq \alpha - \sqrt{\frac{2\log n}{n}} - \frac{2}{n} - \frac{1}{M},\quad 
\Pp{P^n}{\Ecal_{\textnormal{max}}} \geq 1 - \frac{1}{n},\quad 
\Pp{P^n}{\Ecal_{\textnormal{unif}}} \geq 1 - \frac{2}{n},\]
and therefore,
\[\Pp{P^n}{\Ecal_{\textnormal{mod}}\cap \Ecal_{\textnormal{max}}\cap \Ecal_{\textnormal{unif}}}\geq \alpha - 
 \sqrt{\frac{2\log n}{n}} - \frac{1}{M} - \frac{5}{n}.\]
\end{lemma}
\noindent By choosing $M$ to be sufficiently large, then, we obtain
\begin{equation}\label{eqn:three_events}\Pp{P^n}{\Ecal_{\textnormal{mod}}\cap \Ecal_{\textnormal{max}}\cap \Ecal_{\textnormal{unif}}}\geq \alpha - 
 6\sqrt{\frac{\log n}{n}}.\end{equation}

Now it remains to be shown that, for both full conformal prediction and for jackknife+,
we can find an algorithm $\alg$ such that, if
the events $\Ecal_{\textnormal{max}}$, $\Ecal_{\textnormal{mod}}$, and $\Ecal_{\textnormal{unif}}$ all hold,
then the training-conditional miscoverage rate $\alpha_P(\traindata)$ is close to 1.

\begin{proof}[Proof of Theorem~\ref{thm:full_conformal}]
For full conformal, 
 we define a symmetric regression algorithm $\alg$ that maps a data set $\{(x_1,y_1),\dots,(x_{n+1},y_{n+1})\}$ to the fitted function
\[\muh(x) = \begin{cases}
2y_*, & \textnormal{ if } \mod(-a(x) + \sum_{i=1}^{n+1} a(x_i),M)< M_1,\\
0, &\textnormal{ otherwise}.
\end{cases}\]
Below, we will show that, for any training data set $\traindata$,
\begin{equation}\label{eqn:three_events_step}\textnormal{
If $\Ecal_{\textnormal{max}}\cap \Ecal_{\textnormal{mod}}\cap \Ecal_{\textnormal{unif}}$ holds,
then $\Chn(X_{n+1})\subseteq(y_*,\infty)$ almost surely over $X_{n+1}$.}\end{equation}
By definition of $y_*$, we therefore have
\[\Ppst{P}{Y_{n+1} \in\Chn(X_{n+1})}{\traindata}\leq \Ppst{P}{Y_{n+1} \in( y_*,\infty)}{\traindata}= \Pp{P}{Y_{n+1}> y_*}\leq n^{-2},\]
for any $\traindata$ such that $\Ecal_{\textnormal{max}}\cap \Ecal_{\textnormal{mod}}\cap \Ecal_{\textnormal{unif}}$ holds.
Combining this with the bound~\eqref{eqn:three_events}, we have proved that $\Pp{P^n}{\alpha_P(\traindata)\geq 1-n^{-2}}\geq \alpha -  6\sqrt{\frac{\log n}{n}}$, as desired.

To complete the proof, we now verify~\eqref{eqn:three_events_step}. 
Condition on the training data $\traindata$, and assume $\Ecal_{\textnormal{max}}\cap \Ecal_{\textnormal{mod}}\cap \Ecal_{\textnormal{unif}}$ holds.
First, by $\Ecal_{\textnormal{mod}}$, we have
\[\mod\left(-a(X_{n+1}) + \sum_{i=1}^{n+1} a(X_i), M\right) = 
\mod\left(\sum_{i=1}^n a(X_i), M\right)  < M_1\]
for any value of $X_{n+1}$, 
and therefore, for any $y\in\R$,
\[\muh^y_{n+1} (X_{n+1}) = 2y_*\]
where as before, $\muh^y_{n+1} = \alg\big((X_1,Y_1),\dots,(X_n,Y_n),(X_{n+1},y)\big)$.
On the other hand, for any $i\in[n]$, define
\[M_i = \mod\left(-a(X_i) + \sum_{j=1}^{n+1} a(X_j), M\right) = M-1 - \mod\left(a(X_i) - \sum_{j=1}^{n+1} a(X_j) - 1, M\right) ,\]
where the last step holds since $\mod(k,M) = M-1-\mod(-k-1,M)$ for all integers $k$.
By the event $\Ecal_{\textnormal{unif}}$ (applied with $m=- \sum_{j=1}^{n+1} a(X_j) - 1$), 
we see that $M_i \geq M_1$ for at least $\lceil (1-\alpha)(n+1)\rceil$ many $i\in [n]$. Therefore,
for all $y\in\R$,
$\muh^y_{n+1}(X_i) = 0$  for at least $\lceil (1-\alpha)(n+1)\rceil$ many $i\in [n]$.
Next, by $\Ecal_{\textnormal{max}}$,
we have $|Y_i|< y_*$ for all $i\in[n]$, and therefore, for all $y\in\R$, $R^y_i < y_*$  for at least $\lceil (1-\alpha)(n+1)\rceil$ many $i\in [n]$.

Returning to the definition of full conformal prediction given in~\eqref{eqn:define_full_conformal},
we therefore have $\widehat{Q}^y_{n+1} < y_*$. Therefore, $y\in\Chn(X_{n+1})$ can hold
only if $|y - \muh^y_{n+1}(X_{n+1})| < y_*$, which implies $\Chn(X_{n+1})\subseteq (y_*,3y_*)\subseteq (y_*,\infty)$.
This verifies~\eqref{eqn:three_events_step}, and thus completes the proof of the theorem.\end{proof}

\begin{proof}[Proof of Theorem~\ref{thm:jackknife+}]
For jackknife+,
 we define a symmetric regression algorithm $\alg$ that maps a data set $\{(x_1,y_1),\dots,(x_{n-1},y_{n-1})\}$ to the fitted function
\[\muh(x) = \begin{cases}
0, & \textnormal{ if } \mod(a(x) + \sum_{i=1}^{n-1}a(x_i) ,M)< M_1,\\
2y_*, &\textnormal{ otherwise}.
\end{cases}\]
As in the proof of Theorem~\ref{thm:full_conformal}, 
it is sufficient to verify that~\eqref{eqn:three_events_step} again holds in this case.

Condition on the training data $\traindata$, and assume $\Ecal_{\textnormal{max}}\cap \Ecal_{\textnormal{mod}}\cap \Ecal_{\textnormal{unif}}$ holds.
First, by $\Ecal_{\textnormal{mod}}$, for all $i\in[n]$,  we have
\[\mod\left(a(X_i) + \sum_{j \in [n]\backslash\{i\}} a(X_j), M\right) = 
\mod\left(\sum_{j=1}^n a(X_j), M\right)  < M_1,\]
and therefore
\[\muh_{[n]\backslash \{i\}} (X_i) =0.\]
By $\Ecal_{\textnormal{max}}$,
we have $|Y_i|< y_*$ for all $i\in[n]$, and therefore, $R_i = |Y_i - \muh_{[n]\backslash\{i\}}(X_i)| < y_*$  for all $i\in[n]$.

On the other hand, for any $i\in[n]$, define
\[M_i = \mod\left(a(X_{n+1}) + \sum_{j\in[n]\backslash\{i\}} a(X_j), M\right) 
 = \mod\left( - a(X_i) + \sum_{j=1}^{n+1} a(X_j), M\right) .\]
Exactly as in the proof of Theorem~\ref{thm:full_conformal},
by the event $\Ecal_{\textnormal{unif}}$ 
we see that $M_i \geq M_1$ for at least $\lceil (1-\alpha)(n+1)\rceil$ many $i\in [n]$. Therefore,
$\muh_{[n]\backslash\{i\}}(X_{n+1}) = 2y_*$  for at least $\lceil (1-\alpha)(n+1)\rceil$ many $i\in [n]$. 

Combining these calculations, we
see that $\muh_{[n]\backslash\{i\}}(X_{n+1}) - R_i > y_*$ for  at least $\lceil (1-\alpha)(n+1)\rceil$ many $i\in [n]$.
Thus, by definition of the jackknife+ predictive interval given in~\eqref{eqn:define_jackknife+},
we have $\Chn(X_{n+1})\subseteq (y_*, \infty)$.
This verifies~\eqref{eqn:three_events_step}, and thus completes the proof of the theorem.\end{proof}

Finally, we need to prove Lemma~\ref{lem:three_events}.
\begin{proof}[Proof of Lemma~\ref{lem:three_events}]
First, since $y_*$ is chosen to be the $(1-n^{-2})$ quantile of $|Y|$ under the distribution $P$,
we have
\[\Pp{P^n}{\max_{i\in[n]} |Y_i| \geq y_*} \leq \sum_{i=1}^n \Pp{P}{|Y_i|\geq y_*} \leq n\cdot n^{-2} = \frac{1}{n},\]
and thus $\Pp{P^n}{\Ecal_{\textnormal{max}}}\geq 1-\frac{1}{n}$ as desired.

Next, since $a(X_i)\iidsim\textnormal{Unif}\{0,\dots,M-1\}$ for $i\in[n]$,
it follows immediately that $ \mod(a(X_1) + \dots + a(X_n),M) \sim\textnormal{Unif}\{0,\dots,M-1\}$ also, and so
by definition of $M_1$ we have
\[\Pp{P^n}{\Ecal_{\textnormal{mod}}} = \frac{M_1}{M} = \frac{\left\lfloor M\left(\alpha - \sqrt{\frac{2\log n}{n}} - \frac{2}{n}\right)\right\rfloor}{M} 
\geq  \alpha - \sqrt{\frac{2\log n}{n}} -\frac{2}{n}- \frac{1}{M}.\]

Finally, we turn to $\Ecal_{\textnormal{unif}}$. This is the event that
$\sum_{i=1}^n \One{\mod(a(X_i) + m,M)< M - M_1}\geq  \lceil (1-\alpha)(n+1)\rceil$
for all integers $m$, but by definition of the modulo function,
it is equivalent to requiring that this bound holds only for all integers $m=0,\dots,M-1$, i.e.,
\[\Ecal_{\textnormal{unif}} = \bigcap_{m=0}^{M-1} \left\{\sum_{i=1}^n \One{\mod(a(X_i) + m,M)< M - M_1}\geq  \lceil (1-\alpha)(n+1)\rceil \right\}.\]
Now let $U_1,\dots,U_n\iidsim \textnormal{Unif}[0,1]$. Then $\lfloor MU_i\rfloor\iidsim \textnormal{Unif}\{0,\dots,M-1\}$, and so
\begin{equation}\label{eqn:lemma_step}\Pp{P^n}{\Ecal_{\textnormal{unif}} }
=\PP{\bigcap_{m=0}^{M-1} \left\{\sum_{i=1}^n \One{\mod(\lfloor MU_i\rfloor + m,M)< M - M_1}\geq  \lceil (1-\alpha)(n+1)\rceil \right\}}.\end{equation}

Next, suppose it holds that
\begin{equation}\label{eqn:DKW_for_Ecal_unif}\sup_{s\in[0,1]} \left|\sum_i \One{U_i< s} - ns\right| \leq  \sqrt{\frac{n\log n}{2}} .\end{equation}
Then, for each $m\in\{0,1,\dots,M-M_1-1\}$, we have
\begin{align*}
&\sum_{i=1}^n \One{\mod(\lfloor MU_i\rfloor + m,M)< M - M_1}\\
&=\sum_{i=1}^n \One{\lfloor MU_i\rfloor< M-M_1-m} +
\sum_{i=1}^n \One{\lfloor MU_i\rfloor\geq M-m}\\
&=\sum_{i=1}^n \One{U_i< 1-\frac{M_1+m}{M}} + n - 
\sum_{i=1}^n \One{U_i < 1 - \frac{m}{M}}\\
&\geq \left(n\cdot \left(1-\frac{M_1+m}{M}\right) -  \sqrt{\frac{n\log n}{2}}\right) + n -  \left(n \cdot \left(1 - \frac{m}{M}\right) +  \sqrt{\frac{n\log n}{2}}\right)\\
&= n\cdot \left(1 - \frac{M_1}{M}\right) - \sqrt{2n\log n}\\
&\geq \lceil (1-\alpha)(n+1)\rceil,
\end{align*}
where the last step holds by definition of $M_1$.
Next, for each $m\in\{ M-M_1,\dots,M-1\}$, we have
\begin{align*}
&\sum_{i=1}^n \One{\mod(\lfloor MU_i\rfloor + m,M)< M - M_1}\\
&=\sum_{i=1}^n \One{M-m \leq \lfloor MU_i\rfloor< 2M-M_1-m} \\
&=\sum_{i=1}^n \One{U_i< 2-\frac{M_1+m}{M}}  - \sum_{i=1}^n \One{U_i < 1 - \frac{m}{M}} \\
&\geq \left(n\cdot \left(2-\frac{M_1+m}{M}\right) -  \sqrt{\frac{n\log n}{2}}\right) - \left(n \cdot \left(1-\frac{m}{M}\right) + \sqrt{\frac{n\log n}{2}}\right)\\
&= n\cdot \left(1 - \frac{M_1}{M}\right) - \sqrt{2n\log n}\\
&\geq \lceil (1-\alpha)(n+1)\rceil.
\end{align*}
Therefore, returning to~\eqref{eqn:lemma_step}, we have
\[\Pp{P^n}{\Ecal_{\textnormal{unif}} } \geq \PP{\sup_{s\in[0,1]} \left|\sum_i \One{U_i\leq s} - ns\right| \leq  \sqrt{\frac{n\log n}{2}}}
\geq 1 - \frac{2}{n},\]
where the lasts step holds by the Dvoretzky--Kiefer--Wolfowitz inequality. This completes the proof.
\end{proof}

\section{Empirical results}

The theoretical results above suggest that we should be concerned about the training conditional coverage of the full conformal and jackknife+ prediction intervals. However, the algorithms used as counterexamples in the proof are extremely unrealistic. In particular, the
constructions appearing in the proofs display an extremely high amount of instability, since the inclusion of a single training point can greatly impact the output of the regression function. 

Therefore, a natural question is how large the variability of $\alpha_P(\traindata)$ is in practice, particularly in unstable environments.
In our simulation, we will examine the empirical performance of the training-conditional miscoverage rate $\alpha_P(\traindata)$
for the four distribution-free predictive inference tools studied in this work,
for a linear regression task where high dimensionality may cause some instability in the regression 
algorithm.\footnote{Code to reproduce the experiment is available at \url{http://rinafb.github.io/code/training_conditional.zip}.} 

\subsection{Setting}
We choose a target coverage rate of 90\%, i.e., $\alpha = 0.1$,
and will compare the performance of
 split conformal (with $n_0=n_1=n/2$), full conformal, jackknife+, and CV+ (with $K=20$ folds).

We use sample size $n=500$ for the training set, and $n_{\textnormal{test}}=1000$ for the test set.
For each trial, we generate i.i.d.~data points $(X_i,Y_i)$, $i=1,\dots,n+n_{\textnormal{test}}$, from the following distribution:
\[X_i \sim \mathcal{N}(0, I_d) \textnormal{ and } Y_i \mid X_i \sim \mathcal{N}(X_i^\top \beta, 1),\]
where $\beta = \sqrt{10}\cdot U$ for a random unit vector $U$ drawn uniformly from the unit sphere in $\Xcal = \R^d$.
We repeat the experiment for each  dimension $d = 125,250,500,1000$, with $200$ independent trials for each dimension.

Our algorithm $\alg$ is given by ridge regression with penalty parameter $\lambda = 0.0001$, i.e.,
for any data set $(x_1,y_1),\dots,(x_N,y_N)$, the fitted model $\muh = \alg\big((x_1,y_1),\dots,(x_N,y_N)\big)$ is given by
\[\muh(x) = x^\top\widehat\beta\textnormal{ where }\widehat\beta = \arg\min_{\beta\in\R^d}\left\{\sum_{i=1}^N (y_i - x_i^\top\beta)^2 + \lambda\|\beta\|^2_2\right\}.\]
Finally, we estimate the training-conditional miscoverage rate $\alpha_P(\traindata)$ for each of the four methods, by computing
the empirical coverage over the $n_{\textnormal{test}}$ many test points:
\[\alpha_P(\traindata) \approx \frac{1}{n_{\textnormal{test}}}\sum_{i=1}^{n_{\textnormal{test}}} \One{Y_{n+i} \in\Chn(X_{n+i})}.\]

\paragraph{Instability of the algorithm} In this simulation,
 we apply the ridge regression algorithm to a training set where
 the $x_i$'s are standard Gaussian, with a penalty parameter $\lambda\approx 0$.
This optimization problem is extremely poorly conditioned when the training set size is $\approx d$,
 but is well-behaved if the number of training points is {\em either} sufficiently large or sufficiently small relative to $d$
 (see \citet{hastie2022surprises} for an analysis of ``ridgeless'' regression, i.e., taking $\lambda\rightarrow 0$,
  in the overparametrized setting).
As a result, if the number of training points is $\approx d$, 
 the outcome of the algorithm may be highly unstable---the 
 stability assumption~\eqref{eqn:assume_stability} will not hold, and in general, 
predictions $\muh(x)$ will vary greatly with a new draw of the training set. However, instability will
not occur if the training set size  is substantially smaller than $d$ or larger than $d$.

 For split conformal, since the model is trained
on  $n_0 = n/2 = 250$ many data points, this instability will be high for $d=250$ (but not for $d=125,500,1000$). In contrast, when running full conformal
or jackknife+ or CV+,
 the models are trained on $n+1 = 501$ or $n-1=499$ or $n-n/K = 475$ many data points, respectively. Therefore, 
 for these three methods, we expect
instability to be high for dimension $d=500$ (but not for $d=125,250,1000$).

\begin{figure}[!htbp]
	\centering
		\includegraphics[width=0.45\textwidth]{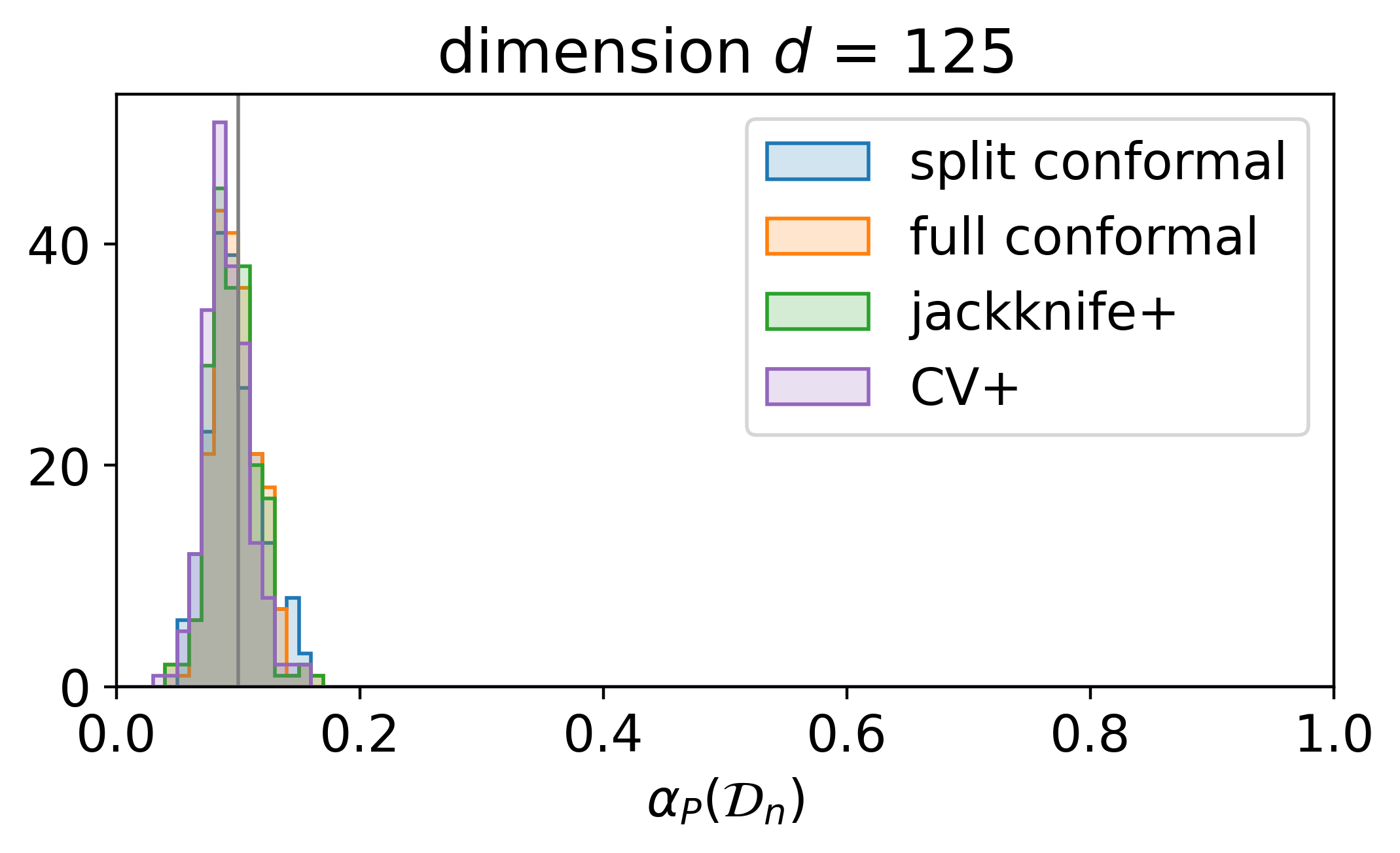}\quad
		\includegraphics[width=0.45\textwidth]{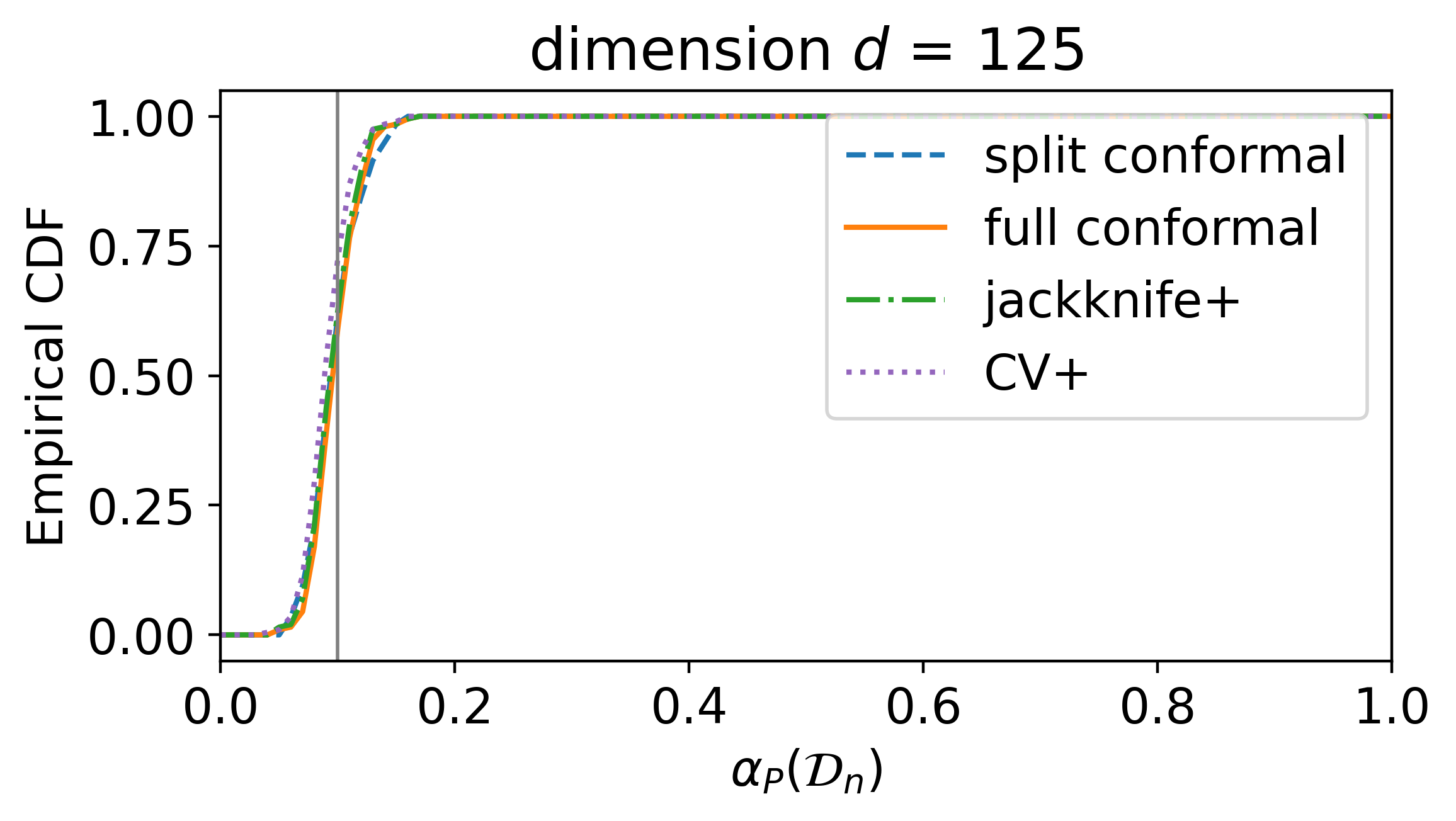}\\\smallskip
		\includegraphics[width=0.45\textwidth]{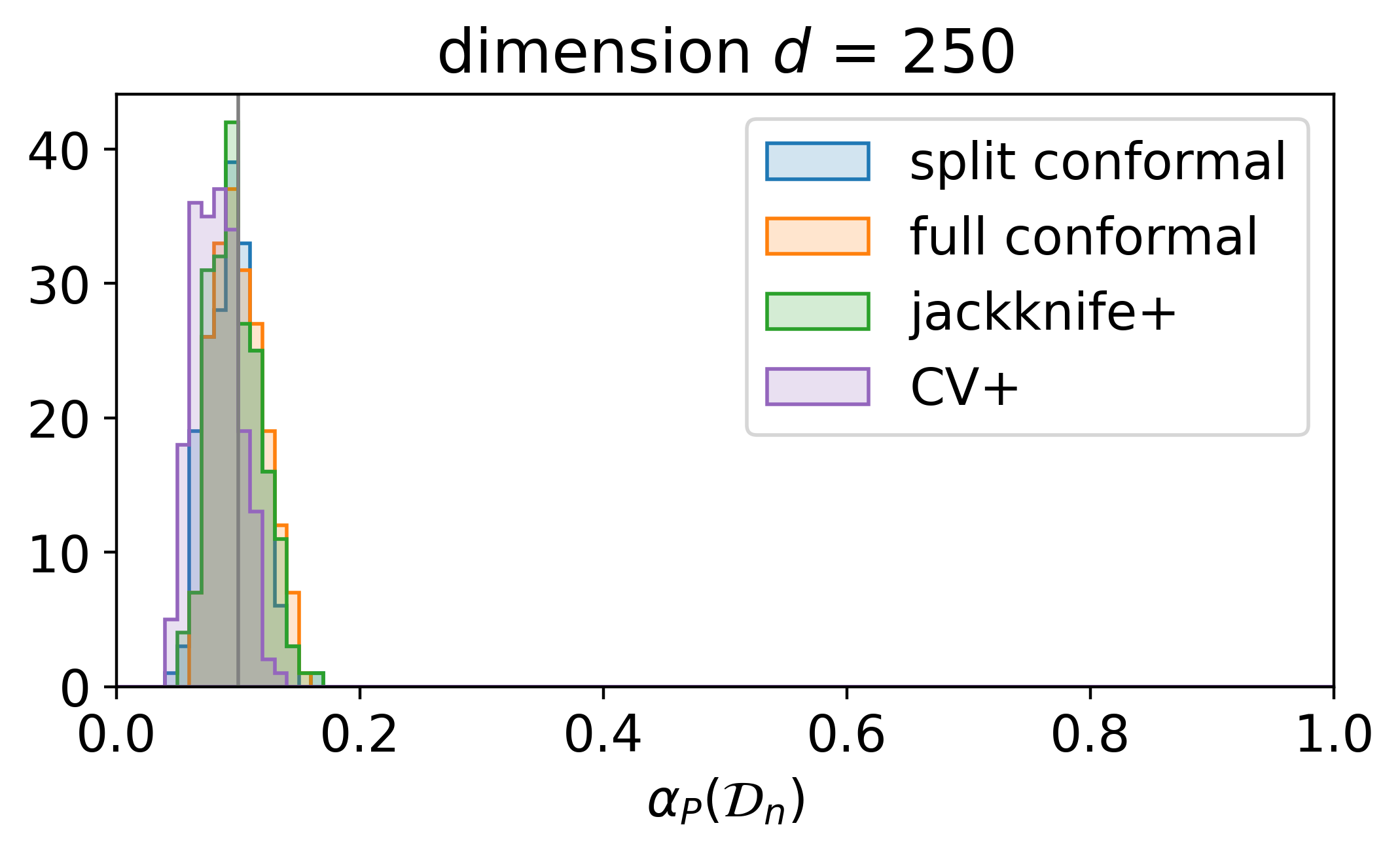}\quad
		\includegraphics[width=0.45\textwidth]{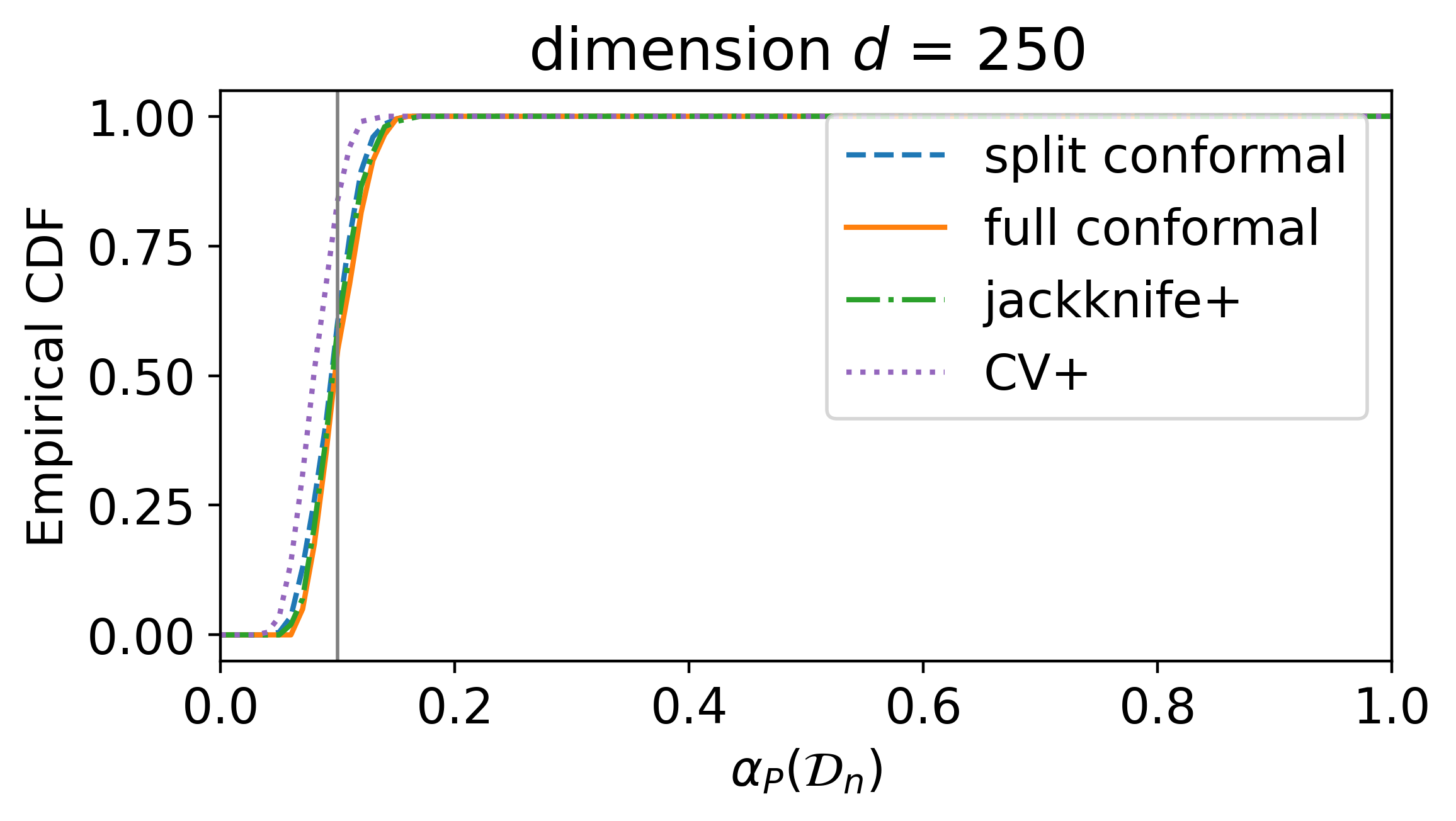}\\\smallskip
		\includegraphics[width=0.45\textwidth]{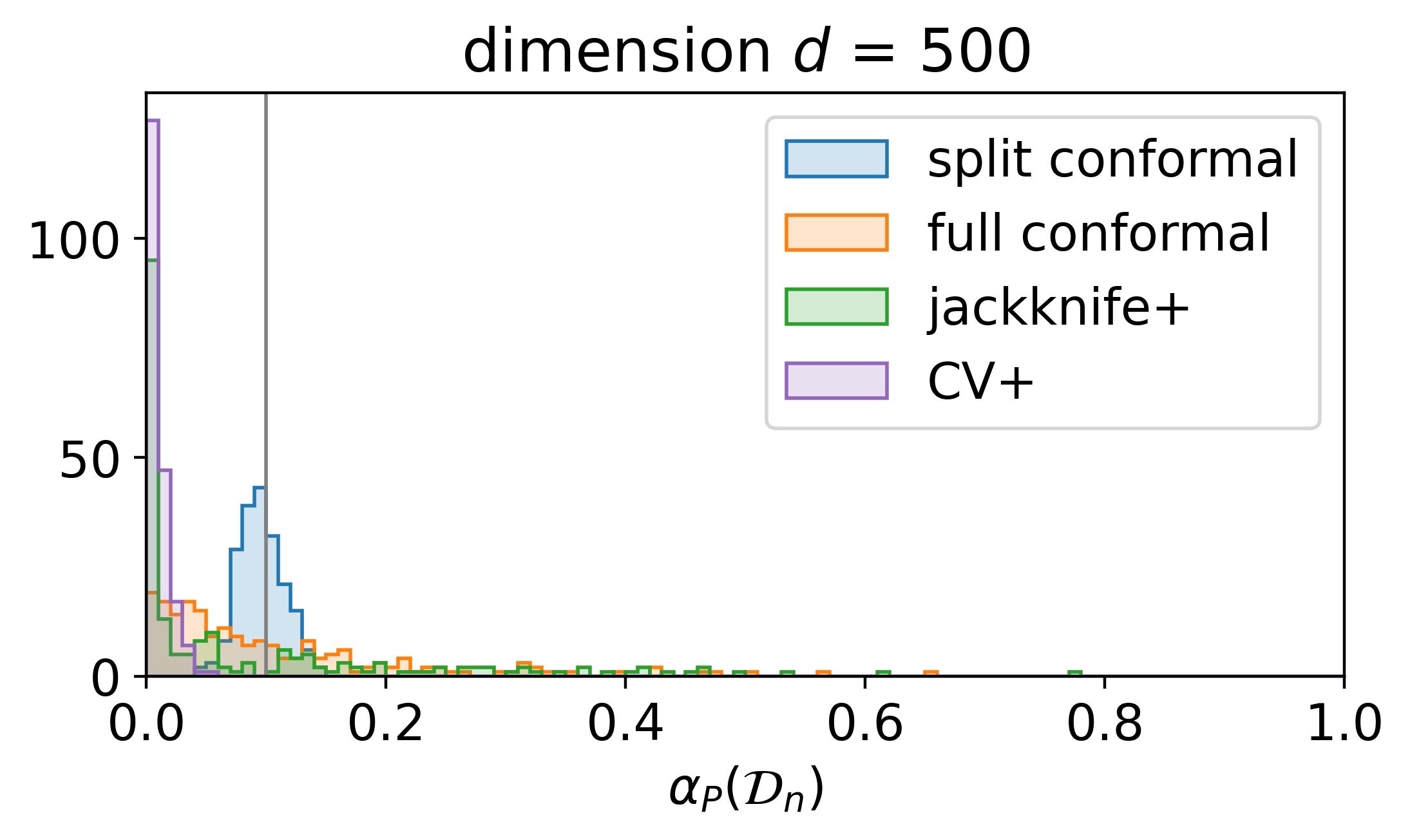}\quad
		\includegraphics[width=0.45\textwidth]{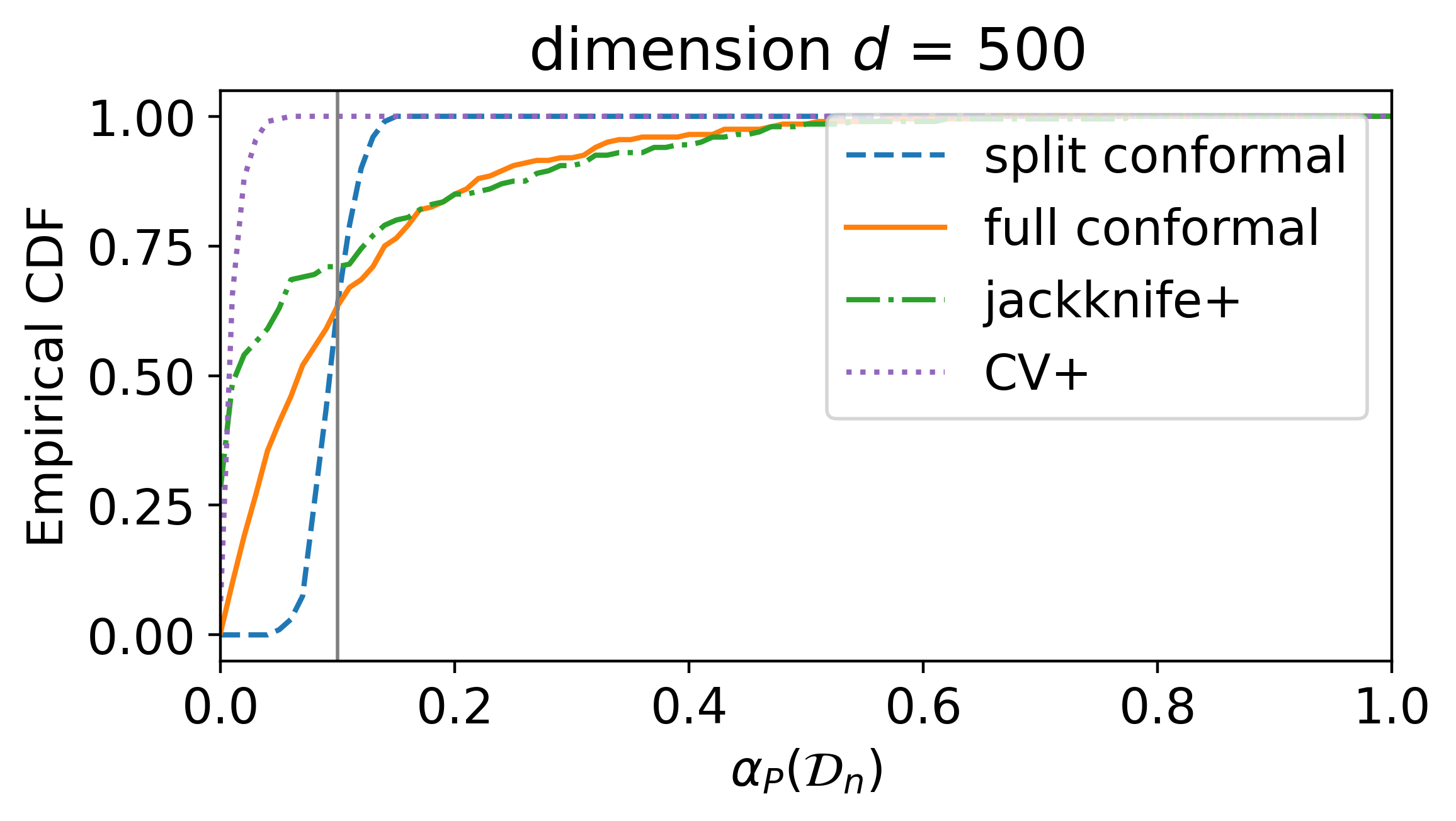}\\\smallskip
		\includegraphics[width=0.45\textwidth]{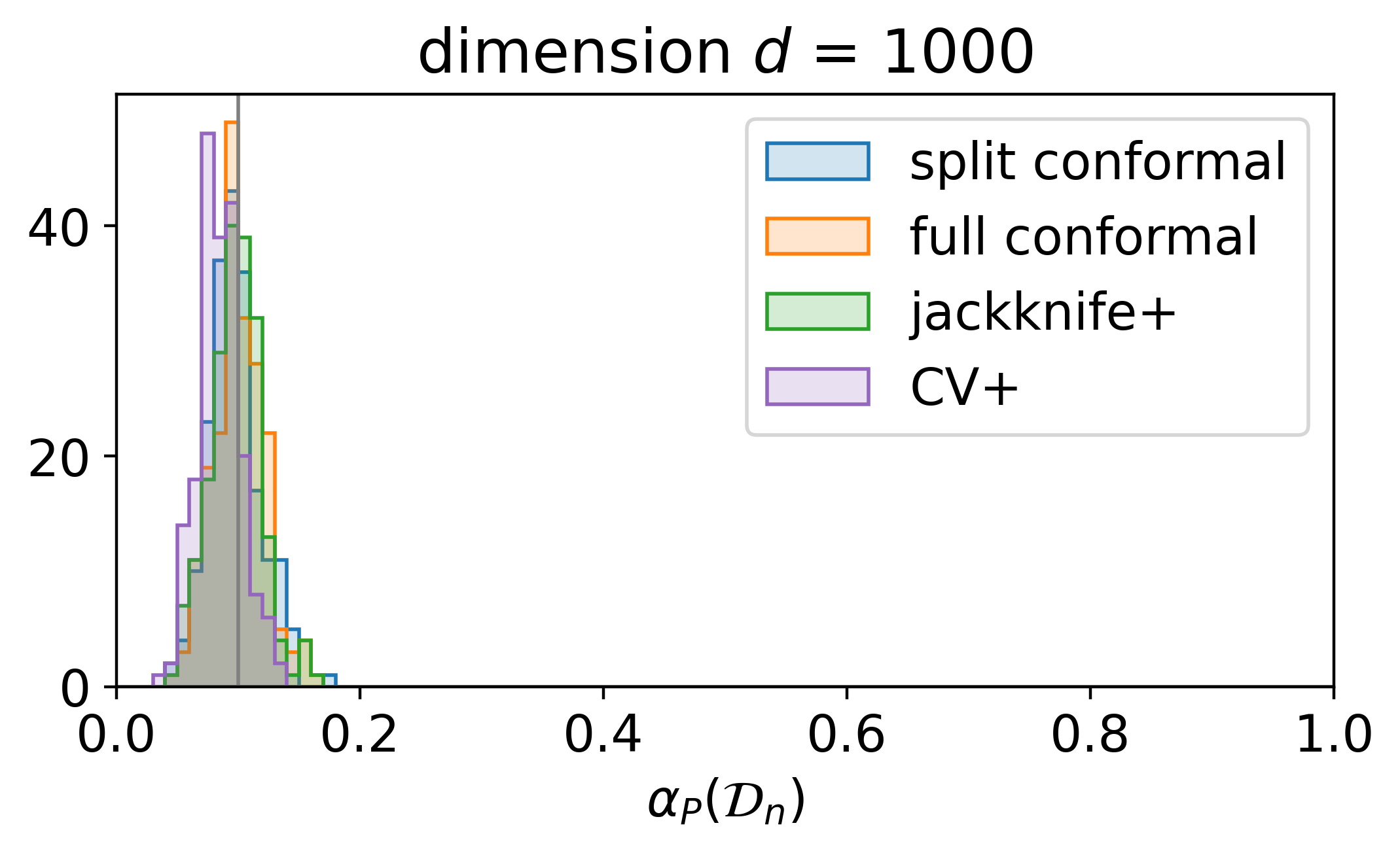}\quad
		\includegraphics[width=0.45\textwidth]{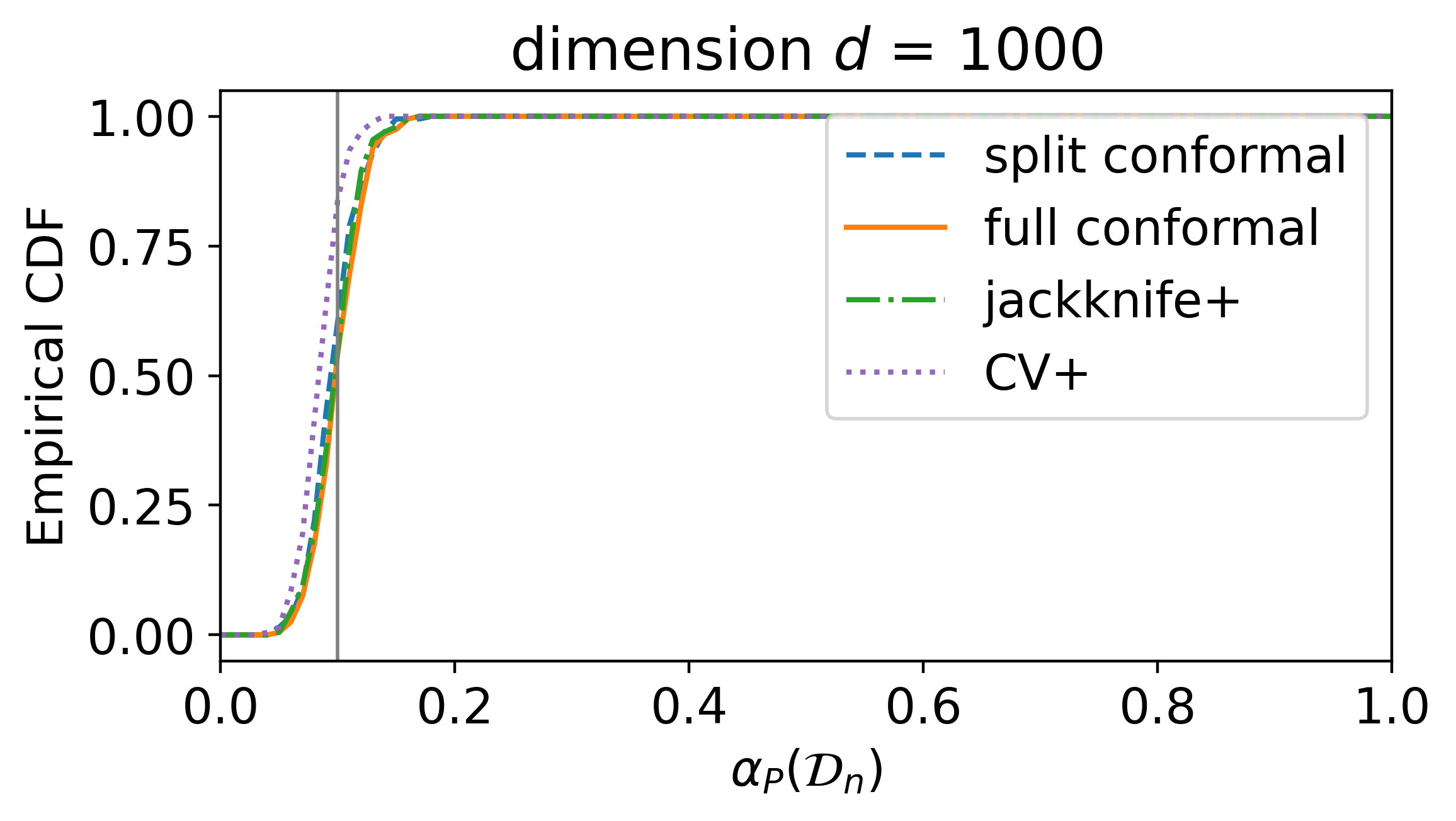}
	\caption{Plots of the estimated training-conditional miscoverage level $\alpha_P(\traindata)$ for
	$200$ independent trials of the simulation, at each dimension $d=125,250,500,1000$. The gray vertical
	lines indicates the target miscoverage level $\alpha = 0.1$. For each $d$, the empirical
	distribution of $\alpha_P(\traindata)$ is displayed as a histogram on the left, and as an empirical CDF
	on the right.}
	\label{fig:sims}
\end{figure}

\subsection{Results}
The results of the simulations are displayed in Figure~\ref{fig:sims}.
For both split conformal and CV+, as the theory suggests, the training-conditional miscoverage rate $\alpha_P(\traindata)$
is consistently near or below the nominal level $\alpha=0.1$. This is the case
even for dimensions where the trained models for split conformal or for CV+ are likely to exhibit instability,
as discussed above. 
Specifically,  for split conformal,  the $\alpha_P(\traindata)$ values
concentrate around $\alpha$ for each choice of dimension $d$.
For CV+, the same is true for dimensions $d=125,250,1000$
where the algorithm is fairly stable,
while in the unstable regime $d=500$, CV+ appears to be 
highly conservative, with $\alpha_P(\traindata)$ values
consistently much lower than $\alpha$. (The same outcome occurs 
if we repeat the experiment with dimension $d=475$, where
instability for CV+ is highest, but for brevity we do not show results for this value of $d$.)

In contrast, for full conformal and for jackknife+,
we see that at $d=500$ (where the trained models for these two methods are likely to be unstable),
the training-conditional miscoverage rate $\alpha_P(\traindata)$ is highly variable---specifically, we see 
that $\alpha_P(\traindata)$ is substantially higher than nominal level $\alpha=0.1$ for a 
large fraction of the trials. On the other hand, the training-conditional miscoverage
rate $\alpha_P(\traindata)$ concentrates near $\alpha = 0.1$ for both methods, for all other values of $d$.
This is true both
for a low-dimensional setting when $d=125,250$ and a high-dimensional (i.e., overparameterized) setting when $d=1000$,
suggesting
that algorithmic stability may play a key role in understanding training-conditional coverage, as we discuss further below.

\section{Conclusion}

In this paper, we examine one form of conditional validity for methods of distribution-free predictive inference: training-conditional validity. While this form of validity has been previously established for the split conformal prediction method,
here we examined whether this property holds for other distribution-free prediction tools. We showed that training conditional coverage guarantees can be ensured for the CV+ method, but are not possible for either the full conformal or jackknife+ methods without
additional assumptions. In addition, we demonstrated empirically that training-conditional miscoverage rates far above the nominal level $\alpha$ can occur in realistic data sets with the latter two methods.

\subsection{The role of algorithmic stability}
An interesting open question is whether there are any mild assumptions that would ensure training-conditional coverage
for full conformal and/or for jackknife+. One possibility is to consider
algorithmic stability assumptions such as~\eqref{eqn:assume_stability}.
In particular, our empirical results show that poor training-conditional coverage for these two methods
is observed exactly in those settings where the behavior of the regression algorithm $\alg$ is highly
unstable (specifically, when $d\approx n$, in our linear regression simulation). This suggests
that assuming stability of $\alg$ could potentially be sufficient to ensure training-conditional coverage
for these methods.
We leave this open question for future work.

\subsection*{Acknowledgements}
R.F.B.~was 
supported by the National Science Foundation via grants DMS-1654076 and DMS-2023109,
and by the Office of Naval Research via grant N00014-20-1-2337. 
The authors are grateful to Ruiting Liang for feedback on an earlier draft of this paper.

\bibliographystyle{plainnat}
\bibliography{bib}

\begin{thebibliography}{25}
\providecommand{\natexlab}[1]{#1}
\providecommand{\url}[1]{\texttt{#1}}
\expandafter\ifx\csname urlstyle\endcsname\relax
  \providecommand{\doi}[1]{doi: #1}\else
  \providecommand{\doi}{doi: \begingroup \urlstyle{rm}\Url}\fi

\bibitem[Barber et~al.(2021{\natexlab{a}})Barber, Cand{\`e}s, Ramdas, and
  Tibshirani]{barber2021limits}
Rina~Foygel Barber, Emmanuel~J Cand{\`e}s, Aaditya Ramdas, and Ryan~J
  Tibshirani.
\newblock The limits of distribution-free conditional predictive inference.
\newblock \emph{Information and Inference: A Journal of the IMA}, 10\penalty0
  (2):\penalty0 455--482, 2021{\natexlab{a}}.

\bibitem[Barber et~al.(2021{\natexlab{b}})Barber, Cand{\`e}s, Ramdas, and
  Tibshirani]{barber2021predictive}
Rina~Foygel Barber, Emmanuel~J Cand{\`e}s, Aaditya Ramdas, and Ryan~J
  Tibshirani.
\newblock Predictive inference with the jackknife+.
\newblock \emph{The Annals of Statistics}, 49\penalty0 (1):\penalty0 486--507,
  2021{\natexlab{b}}.

\bibitem[Bates et~al.(2021)Bates, Angelopoulos, Lei, Malik, and
  Jordan]{bates2021distribution}
Stephen Bates, Anastasios Angelopoulos, Lihua Lei, Jitendra Malik, and Michael
  Jordan.
\newblock Distribution-free, risk-controlling prediction sets.
\newblock \emph{Journal of the ACM (JACM)}, 68\penalty0 (6):\penalty0 1--34,
  2021.

\bibitem[Chen et~al.(2018)Chen, Chun, and Barber]{chen2018discretized}
Wenyu Chen, Kelli-Jean Chun, and Rina~Foygel Barber.
\newblock Discretized conformal prediction for efficient distribution-free
  inference.
\newblock \emph{Stat}, 7\penalty0 (1):\penalty0 e173, 2018.

\bibitem[Dudley and Norvai{\v{s}}a(2011)]{dudley2011concrete}
Richard~M Dudley and Rimas Norvai{\v{s}}a.
\newblock \emph{Concrete functional calculus}.
\newblock Springer, 2011.

\bibitem[Hastie et~al.(2022)Hastie, Montanari, Rosset, and
  Tibshirani]{hastie2022surprises}
Trevor Hastie, Andrea Montanari, Saharon Rosset, and Ryan~J Tibshirani.
\newblock Surprises in high-dimensional ridgeless least squares interpolation.
\newblock \emph{The Annals of Statistics}, 50\penalty0 (2):\penalty0 949--986,
  2022.

\bibitem[Kivaranovic et~al.(2020)Kivaranovic, Johnson, and
  Leeb]{kivaranovic2020adaptive}
Danijel Kivaranovic, Kory~D Johnson, and Hannes Leeb.
\newblock Adaptive, distribution-free prediction intervals for deep networks.
\newblock In \emph{International Conference on Artificial Intelligence and
  Statistics}, pages 4346--4356. PMLR, 2020.

\bibitem[Lei(2019)]{lei2019fast}
Jing Lei.
\newblock Fast exact conformalization of the lasso using piecewise linear
  homotopy.
\newblock \emph{Biometrika}, 106\penalty0 (4):\penalty0 749--764, 2019.

\bibitem[Lei and Wasserman(2014)]{lei2014distribution}
Jing Lei and Larry Wasserman.
\newblock Distribution-free prediction bands for non-parametric regression.
\newblock \emph{Journal of the Royal Statistical Society: Series B (Statistical
  Methodology)}, 76\penalty0 (1):\penalty0 71--96, 2014.

\bibitem[Lei et~al.(2018)Lei, G'Sell, Rinaldo, Tibshirani, and
  Wasserman]{lei2018distribution}
Jing Lei, Max G'Sell, Alessandro Rinaldo, Ryan~J Tibshirani, and Larry
  Wasserman.
\newblock Distribution-free predictive inference for regression.
\newblock \emph{Journal of the American Statistical Association}, 113\penalty0
  (523):\penalty0 1094--1111, 2018.

\bibitem[L{\"o}fstr{\"o}m et~al.(2015)L{\"o}fstr{\"o}m, Bostr{\"o}m, Linusson,
  and Johansson]{lofstrom2015bias}
Tuve L{\"o}fstr{\"o}m, Henrik Bostr{\"o}m, Henrik Linusson, and Ulf Johansson.
\newblock Bias reduction through conditional conformal prediction.
\newblock \emph{Intelligent Data Analysis}, 19\penalty0 (6):\penalty0
  1355--1375, 2015.

\bibitem[Park et~al.(2020)Park, Li, Lee, and Bastani]{park2020pac}
Sangdon Park, Shuo Li, Insup Lee, and Osbert Bastani.
\newblock Pac confidence predictions for deep neural network classifiers.
\newblock \emph{arXiv preprint arXiv:2011.00716}, 2020.

\bibitem[Park et~al.(2021)Park, Dobriban, Lee, and Bastani]{park2021pac}
Sangdon Park, Edgar Dobriban, Insup Lee, and Osbert Bastani.
\newblock {PAC} prediction sets under covariate shift.
\newblock \emph{arXiv preprint arXiv:2106.09848}, 2021.

\bibitem[Qiu et~al.(2022)Qiu, Dobriban, and Tchetgen]{qiu2022distribution}
Hongxiang Qiu, Edgar Dobriban, and Eric~Tchetgen Tchetgen.
\newblock Distribution-free prediction sets adaptive to unknown covariate
  shift.
\newblock \emph{arXiv preprint arXiv:2203.06126}, 2022.

\bibitem[R{\"u}schendorf(1982)]{ruschendorf1982random}
Ludger R{\"u}schendorf.
\newblock Random variables with maximum sums.
\newblock \emph{Advances in Applied Probability}, 14\penalty0 (3):\penalty0
  623--632, 1982.

\bibitem[Steinberger and Leeb(2018)]{steinberger2018conditional}
Lukas Steinberger and Hannes Leeb.
\newblock Conditional predictive inference for high-dimensional stable
  algorithms.
\newblock \emph{arXiv preprint arXiv:1809.01412}, 2018.

\bibitem[Vovk(2012)]{vovk2012conditional}
Vladimir Vovk.
\newblock Conditional validity of inductive conformal predictors.
\newblock In \emph{Asian conference on machine learning}, pages 475--490. PMLR,
  2012.

\bibitem[Vovk(2015)]{vovk2015cross}
Vladimir Vovk.
\newblock Cross-conformal predictors.
\newblock \emph{Annals of Mathematics and Artificial Intelligence}, 74\penalty0
  (1):\penalty0 9--28, 2015.

\bibitem[Vovk and Wang(2020)]{vovk2020combining}
Vladimir Vovk and Ruodu Wang.
\newblock Combining p-values via averaging.
\newblock \emph{Biometrika}, 107\penalty0 (4):\penalty0 791--808, 2020.

\bibitem[Vovk et~al.(2005)Vovk, Gammerman, and Shafer]{vovk2005algorithmic}
Vladimir Vovk, Alexander Gammerman, and Glenn Shafer.
\newblock \emph{Algorithmic learning in a random world}.
\newblock Springer Science \& Business Media, 2005.

\bibitem[Vovk et~al.(2018)Vovk, Nouretdinov, Manokhin, and
  Gammerman]{vovk2018cross}
Vladimir Vovk, Ilia Nouretdinov, Valery Manokhin, and Alexander Gammerman.
\newblock Cross-conformal predictive distributions.
\newblock In \emph{Conformal and Probabilistic Prediction and Applications},
  pages 37--51. PMLR, 2018.

\bibitem[Wald(1943)]{wald1943extension}
Abraham Wald.
\newblock An extension of wilks' method for setting tolerance limits.
\newblock \emph{The Annals of Mathematical Statistics}, 14\penalty0
  (1):\penalty0 45--55, 1943.

\bibitem[Wilks(1941)]{wilks1941determination}
Samuel~S Wilks.
\newblock Determination of sample sizes for setting tolerance limits.
\newblock \emph{The Annals of Mathematical Statistics}, 12\penalty0
  (1):\penalty0 91--96, 1941.

\bibitem[Yang and Kuchibhotla(2021)]{yang2021finite}
Yachong Yang and Arun~Kumar Kuchibhotla.
\newblock Finite-sample efficient conformal prediction.
\newblock \emph{arXiv preprint arXiv:2104.13871}, 2021.

\bibitem[Yang et~al.(2022)Yang, Kuchibhotla, and Tchetgen]{yang2022doubly}
Yachong Yang, Arun~Kumar Kuchibhotla, and Eric~Tchetgen Tchetgen.
\newblock Doubly robust calibration of prediction sets under covariate shift.
\newblock \emph{arXiv preprint arXiv:2203.01761}, 2022.

\end{thebibliography}

\end{document}